\documentclass[12pt, a4paper]{amsart}

\usepackage{amsmath,amsfonts,amsthm,amssymb,indentfirst,epic}

\usepackage{hyperref}

\newtheorem{thm}{Theorem}[section]

\newtheorem{lem}[thm]{Lemma}

\newtheorem{prop}[thm]{Proposition}

\newtheorem{question}[thm]{Question}

\newtheorem{defin}[thm]{Definition}

\newtheorem{remark}[thm]{Remark}

\def\Q {{\mathbb Q}}

\def\N{{\mathbb N}}
\def\Z {{\mathbb Z}}

\newcommand\Char{{\operatorname{Char}}}

\newcommand\sg[1]{{\left<{#1}\right>}}
\newcommand\Supp{{\operatorname{Supp}}}

\DeclareMathOperator{\Aut}{Aut}
\DeclareMathOperator{\End}{End}

\newcommand{\ben}{\begin{enumerate}}
\newcommand{\een}{\end{enumerate}}
\newcommand{\bit}{\begin{itemize}}
\newcommand{\eit}{\end{itemize}}

\begin{document}
\title[]
{\large{The Starred Dixmier's conjecture}
 }
\author[]
{\small{Vered Moskowicz}}

\begin{abstract}
Dixmier's famous question says the following: Is every algebra endomorphism of the first Weyl algebra, $A_1(F)$, $\Char(F)= 0$, an automorphism?

Let $\alpha$ be the exchange involution on $A_1(F)$: $\alpha(x)= y$, $\alpha(y)= x$. An $\alpha$-endomorphism of $A_1(F)$ is an endomorphism which preserves the involution $\alpha$. 

Then one may ask the following question which may be called the ``$\alpha$-Dixmier's problem $1$" or the ``starred Dixmier's problem $1$": Is every $\alpha$-endomorphism of $A_1(F)$, $\Char(F)= 0$, an automorphism? 
\end{abstract}

\maketitle

\section{Introduction}
By definition, the $n$'th Weyl algebra $A_n(F)= A_n$ is the unital associative $F$-algebra generated by $2n$ elements $x_1, \ldots, x_n, y_1, \ldots, y_n$ subject to the following defining relations: 
$[y_i,x_j]= \delta_{ij}$, $[x_i,x_j]= 0$ and $[y_i,y_j]= 0$, where $\delta_{ij}$ is the Kronecker delta. 

In \cite{adja} Adjamagbo and van den Essen remarked that $A_1$ was first studied by Dirac in \cite{dirac}. Hence, they suggest to call it ``Dirac quantum algebra" instead of (first) Weyl algebra. Similarly, they suggest to call $A_n$  ``$n$'th Dirac quantum algebra" instead of $n$'th Weyl algebra. We truely do not know which name is better. For convenience, we shall continue to call $A_1$ the first Weyl algebra (and $A_n$ the $n$'th Weyl algebra).

In \cite{dixmier}, Dixmier asked six questions about the first Weyl algebra $A_1(F)$, where $F$ is a zero characteristic field; the first question is the following: Is every algebra endomorphism of $A_1(F)$ an automorphism? 

It is well known that over a zero characteristic field $F$, $A_1(F)$ is simple, hence every algebra endomorphism of $A_1(F)$ is necessarily injective (one-one). Therefore, Dixmier's first question can be rephrased as: Is every algebra endomorphism of $A_1(F)$ onto?

Usually, Dixmier's first question is brought as a conjecture; namely, Dixmier's conjecture says that every algebra endomorphism of $A_1(F)$ is an automorphism (= every algebra endomorphism of $A_1(F)$ is onto).

\section{The starred Dixmier's conjecture seems to be true, at least in some special cases}
We denote by $A_1$ the first Weyl algebra $A_1(F)= F\sg{x,y | yx-xy= 1}$, where $\Char(F)= 0$.

Obviously, for a mapping $f: A_1 \longrightarrow A_1$ to be an $F$-algebra homomorphism (endomorphism), it is enough that $[f(y),f(x)]= f(y)f(x)- f(x)f(y)= 1$.
Also, for a mapping $f: A_1 \longrightarrow A_1$ to be an $F$-algebra antihomomorphism (=antiendomorphism, namely, for every $a,b \in A_1$, $f(ab)= f(b)f(a)$), it is enough that $[f(y),f(x)]= f(y)f(x)- f(x)f(y)= -1$.

Notice that the following mapping $\alpha: A_1 \longrightarrow A_1$ is an involution on $A_1$: 
$\alpha(x)= y, \alpha(y)= x$. Indeed, $\alpha$ is an antihomomorphism of order $2$, so it is an antiautomorphism of order $2$. This mapping $\alpha$ is sometimes called the exchange involution.

Of course, there are other involutions on $A_1$. For example, given any automorphism $g$ of $A_1$, $g^{-1} \alpha g$ is clearly an involution on $A_1$. 

Generally, it is easy to see that each involution on $A_1$ is of the form $h \alpha$, where $h$ is an automorphism of $A_1$ which satisfies the following condition $h \alpha h \alpha= 1$ ($1$ is the identity map).
Indeed, let $\beta$ be any involution on $A_1$. Then $\beta \alpha$ is an automorphism of $A_1$, call it $h$. From $\beta \alpha= h$ follows $\beta= h \alpha$. Of course, since $\beta^2= 1$, we get $h \alpha h \alpha= 1$.

\begin{defin}\label{defin alpha endo}
An $\alpha$-endomorphism of $A_1$, $f$, is an endomorphism of $A_1$ which preserves the involution $\alpha$. Preserving the involution $\alpha$ means that for every $w \in A_1$, $f(\alpha(w))= \alpha(f(w))$.
So an $\alpha$-endomorphism of $A_1$, $f$, is an endomorphism of $A_1$ which commutes with $\alpha$ ($f\alpha= \alpha f$).
\end{defin}

It is easy to see that for every $w \in A_1$: $f(\alpha(w))= \alpha(f(w))$ $\Leftrightarrow$ $f(\alpha(x))= \alpha(f(x))$ and $f(\alpha(y))= \alpha(f(y))$ (write $w= \sum \alpha_{ij} x^iy^j$, $\alpha_{ij} \in F$. A direct computation shows that $f(\alpha(w))= \alpha(f(w))$, given $f(\alpha(x))= \alpha(f(x))$ and $f(\alpha(y))= \alpha(f(y))$).
Therefore, $f$ is an $\alpha$-endomorphism of $A_1$, if $f$ is an endomorphism of $A_1$, in which $f(\alpha(x))= \alpha(f(x))$ and $f(\alpha(y))= \alpha(f(y))$.

More generally, given any two involutions $\beta$ and $\gamma$ on $A_1$, one can define a $(\beta, \gamma)$-endomorphism of $A_1$ as an endomorphism of $A_1$ which preserves the involutions $\beta$ and $\gamma$. Preserving the involutions $\beta$ and $\gamma$ means that for every $w \in A_1$, $f(\beta(w))= \gamma(f(w))$.
If $\beta= \gamma$, then a $(\beta, \gamma)$-endomorphism is just a $\beta$-endomorphism.

However, we will only deal with the exchange involution $\alpha$; hence when we consider symmetric or antisymmetric elements of $A_1$, we mean with respect to $\alpha$.

Notice that the set of symmetric elements is $F$-linearly spanned by $\{x^ny^m + x^my^n | n \geq m \}$, 
while the set of antisymmetric elements is $F$-linearly spanned by $\{x^ny^m - x^my^n | n > m \}$.
Clearly, $\{x^ny^m + x^my^n | n \geq m \} \cup  \{x^ny^m - x^my^n | n > m \}$ is a basis of $A_1$ as a vector space over $F$.

Now, one may ask the ``$\alpha$-Dixmier's problem $1$" or the ``starred Dixmier's problem $1$": Is every $\alpha$-endomorphism of $A_1(F)$ ($\Char(F)=0$) an automorphism? 
Also, one may pose the ``$\alpha$-Dixmier's conjecture" or the ``starred Dixmier's conjecture": Every $\alpha$-endomorphism of $A_1(F)$ ($\Char(F)=0$) is an automorphism.

\begin{remark}
The exchange involution $\alpha$ may be denoted by ``$*$ on the right", instead of ``$\alpha$ on the left" (namely, $x^*= y$ and $y^*= x$ instead of $\alpha(x)= y$ and $\alpha(y)= x$), hence the name the ``starred Dixmier's problem $1$".
\end{remark}

In Theorem \ref{starred true}, we hopefully show that the starred Dixmier's conjecture is true in some special cases.
Our proof relies heavily on Joseph's results (especially \cite[Corollary 5.5]{joseph}). 

We will use some basic notions concerning $A_1(F)$, which can be found in \cite{dixmier}, \cite{joseph} and \cite{shape of possible counterexamples}.

We will need to extend $A_1= A_1(F)$ to $A_1^{(m)}$ ($m$ is a positive integer) as Joseph did in \cite[Section 3]{joseph}, where $A_1^{(m)}$ is the unital associative $F$-algebra generated by $y, x, x^{- 1/m}$.

Recall the following notions (for more details, see \cite[Definition 1.1, Notations 1.2 and Notations 1.3]{shape of possible counterexamples}):

Let $F[X,Y]$ be the commutative polynomial algebra in two variables and let 
$\Theta: A_1\longrightarrow F[X,Y]$ be the $F$-linear map defined by $\Theta(\sum \alpha_{ij}x^iy^j):= \sum \alpha_{ij}X^iY^j$.

Let $E=: \{(r,s) \in (\Z)^2 | \gcd(r,s)=1 , r+s \geq 0\}$.

For all $(r,s) \in E$ and $(i,j) \in (\N)^2$ (in view of Joseph's extension, we will need to consider also $(i,j) \in (\N)\times(\Z)$), let $\nu_{r,s}(i,j):= ri+sj$.

Fix $(r,s) \in E$. 
For $0 \neq P= \sum \alpha_{ij}X^iY^j \in F[X,Y]$, we define:
\bit
\item The support of $P$, $\Supp(P)$, as $\{(i,j) | \alpha_{ij}\neq 0\}$ (The support of $P$ is, of course, independent of the choice of $(r,s)$).
\item The $(r,s)$ degree of $P$, $\nu_{r,s}(P)$, as the maximum of the set 
$\{ \nu_{r,s}(i,j) | \alpha_{ij}\neq 0\}= \{ \nu_{r,s}(i,j) | (i,j)\in \Supp(P)\}$. 
\item The $(r,s)$ leading term of $P$, $l_{r,s}(P)$, 
as $\sum_{\nu_{r,s}(i,j)= \nu_{r,s}(P)} \alpha_{ij}X^iY^j$.
\eit

For example, if $(r,s)=(1,1)$ and $P= 2+7XY+20X^3Y^5-X^4Y^4$, then: $\Supp(P)=\{(0,0),(1,1),(3,5),(4,4)\}$, 
$\nu_{1,1}(P)= 8$ and $l_{r,s}(P)= 20X^3Y^5-X^4Y^4$.

For $0 \neq u \in A_1$, we define:
\bit
\item The support of $u$, $\Supp(u)$, as $Supp(\Theta(u))$.
\item The $(r,s)$ degree of $u$, $\nu_{r,s}(u)$, as $\nu_{r,s}(\Theta(u))$.
\item The $(r,s)$ leading term of $u$, $l_{r,s}(u)$, as $l_{r,s}(\Theta(u))$.
\item $w(u)= (i_0,i_0-\nu_{1,-1}(u))$, where 
$i_0$ is the maximum of the first component from pairs appearing in the support of $l_{1,-1}(u)$.
\eit

Our notations will be as in \cite[Notations 1.2, 1.3, 1.4]{shape of possible counterexamples}. Notice that only $A_1$ is considered in that paper. We will use those notations also for $A_1^{(m)}$. 
In particular, for $0 \neq u \in A_1^{(m)}$, the $(r,s)$ degree of $u$ will be denoted by $\nu_{r,s}(u)$ and the $(r,s)$ leading term of $u$ will be denoted by $l_{r,s}(u)$. 
Here, $(r,s)$ will be $(1, -1)$, $(1, 1)$ or $(1, 0)$.

For more details about those notations for $A_1^{(m)}$ (and for additional results), see \cite{shape of possible counterexamples version 3}.

\begin{prop}\label{degree}
Let $u,v \in A_1^{(m)} \ 0$. Then $\nu_{1, -1}(uv)= \nu_{1, -1}(u)+ \nu_{1, -1}(v)$.
\end{prop}

\begin{proof}
This is \cite[Proposition 1.9(3)]{shape of possible counterexamples}, extended to $A_1^{(m)}$.

One has to replace $x^{l/m}y$ by $yx^{l/m}- (l/m)x^{l/m - 1}$ ($l \in \Z$), see \cite[page 603]{joseph} (the four lines between Proposition $3.2$ and Proposition $3.3$). Notice that there $\rho+ \sigma >0$ while here $\rho= 1, \sigma= -1$, and here both the first term and the second term has unaltered $(1, -1)$ degree.
\end{proof}

\begin{defin}
We say that an $\alpha$-endomorphism $f$ of $A_1$ has degree $n$, if the $(1, 1)$ degree of $f(x)$ is $n$ 
($\nu_{1, 1}(f(x))= n$).
\end{defin}

\begin{remark}
Obviously, if $f$ is an $\alpha$-endomorphism of $A_1$, then: 
The $(1, 1)$ degree of $f(x)$ is $n$ $\Leftrightarrow$ The $(1, 1)$ degree of $f(y)$ is $n$.
\end{remark}

The following useful lemma relies on \cite[Theorem 1.22(2)]{shape of possible counterexamples}. It will be used in the proof of Theorem \ref{starred true} and in the proof of Theorem \ref{prime degree}.

\begin{lem}\label{useful lemma}
Let $f$ be an $\alpha$-endomorphism of $A_1$ of degree $> 1$. Let $f(x)= P$ (hence $f(y)= P^*$). 
Then the $(1, 1)$ leading term of $P$, $l_{1, 1}(P)$, is symmetric or antisymmetric.
\end{lem}

\begin{remark}\label{degree 1 is auto}
It is easy to see that when $f$ is an $\alpha$-endomorphism of $A_1$ of degree $1$, 
then $f(x)= Ax+By+C, f(y)= Bx+Ay+C$, with $A^2- B^2= 1$. The $(1, 1)$ leading term of $f(x)$, which is $Ax+By$, is not symmetric or antisymmetric. 

Clearly, such $f$ is an automorphism.
\end{remark}

\begin{proof}
Let $n >1$ be the degree of $f$. Then we can write $P= E_n+ E_{n-1}+ \ldots + E_1+ E_0$, where $E_n \neq 0$ and $E_i$ ($0 \leq i \leq n$) is $(1, 1)$ homogeneous of $(1, 1)$ degree $i$. Since $f$ is an $\alpha$-endomorphism of $A_1$, 
$1= [f(y), f(x)]= [P^*, P]$.

We can apply \cite[Theorem 1.22(2)]{shape of possible counterexamples}. Indeed, $[P, P^*]= -1$, 
hence $[P, P^*]_{1, 1}= 0$, since $\nu_{1, 1}([P, P^*])= \nu_{1, 1}(-1)= 0 < \nu_{1, 1}(P) +\nu_{1, 1}(P^*) - 2$ (we assumed that $f$ is of degree $> 1$. Also, we have remarked that $\nu_{1, 1}(P)= \nu_{1, 1}(P^*)$).

Therefore, there exist $0 \neq \lambda \in F$ and $0 \neq \mu \in F$, $M,N \in \N$ with $\gcd(M,N)= 1$ and a $(1, 1)$ homogeneous polynomial $R \in F[X,Y]$ ($F[X,Y]$ is the commutative polynomial $F$-algebra), such that:
$M/N= \nu_{1, 1}(P) / \nu_{1, 1}(P^*)$, the $(1, 1)$ leading term of $P$ is $\lambda R^M$ and the $(1, 1)$ leading term of $P^*$ is $\mu R^N$. 

(Remark: when we write $R^M$ as an element of $A_1$, we mean the computation of $R^M$ in $F[X,Y]$ and then replacing $X,Y$ by $x,y$. For example, if $R= X- Y$, then $R^2$ as an element of $A_1$ is $x^2- 2xy+ y^2$). 

{}From $\nu_{1, 1}(P)= \nu_{1, 1}(P^*)= n$, we get $M/N= \nu_{1, 1}(P) / \nu_{1, 1}(P^*)= n/n= 1$, so $M=N$.
But $1= \gcd(M,N)= \gcd(M,M)= M$, so $M= N= 1$.
Hence, the $(1, 1)$ leading term of $P$ is $\lambda R$ and the $(1, 1)$ leading term of $P^*$ is $\mu R$. 

Of course, the $(1, 1)$ leading term of $P$ is $E_n$ and the leading $(1, 1)$ term of $P^*$ is $(E_n)^*$. 
Therefore, $E_n= \lambda R$ and $(E_n)^*= \mu R$.

Hence, $\mu R= (E_n)^*= (\lambda R)^*$, and we have $\mu R= (\lambda R)^*= \lambda(R^*)$, 
so $R^*= \lambda^{-1} \mu R$.

Write $R=S+ K$, where $S$ is symmetric and $K$ is antisymmetric 
-just take $S= (R+ R^*)/2$ and $K=(R- R^*)$.

Hence, $(S+K)^*= \lambda^{-1} \mu (S+K)$, so
$S- K= \lambda^{-1} \mu (S+ K)$. Then $S- K= \lambda^{-1} \mu S+ \lambda^{-1} \mu K$, 
implying $S- \lambda^{-1} \mu S= K+ \lambda^{-1} \mu K$, so
$(1- \lambda^{-1} \mu) S= (1+ \lambda^{-1} \mu) K$. Since an element which is both symmetric and antisymmetric must be the zero element, we get $(1- \lambda^{-1} \mu) S= 0$ and $(1+ \lambda^{-1} \mu) K= 0$.

\bit
\item If $K= 0$, then $R= S$. $R \neq 0$ (otherwise, $R=0$, hence $E_n= \lambda R= 0$, a contradiction), 
so $S \neq 0$, which implies that $(1- \lambda^{-1} \mu)= 0$. Therefore, $\mu= \lambda$. 
Hence we have, $E_n= \lambda R= \lambda S$ and $(E_n)^*= \mu R= \mu S= \lambda S$, which shows that $E_n$ is symmetric.

\item If $S= 0$, then $R= K$. $R \neq 0$ (otherwise, $R=0$, hence $E_n= \lambda R= 0$, a contradiction), 
so $K \neq 0$, which implies that $(1+ \lambda^{-1} \mu)= 0$. Therefore, $\mu= -\lambda$. 
Hence we have, $E_n= \lambda R= \lambda K$ and $(E_n)^*= \mu R= \mu K= -\lambda K$, which shows that $E_n$ is antisymmetric.

\item If $K \neq 0$ and $S \neq 0$, then $1+ \lambda^{-1} \mu= 0$ and $1- \lambda^{-1} \mu= 0$. 
Then $\mu= -\lambda$ and $\mu= \lambda$, which is impossible (since $\lambda \neq 0$).
\eit

\end{proof}
As a first step, one may wish to find an example of an $\alpha$-endomorphism of $A_1$ of degree $> 1$.

In the following proposition we describe a family of $\alpha$-endomorphisms of any even degree (this family also includes degree $1$ endomorphisms),
which is actually a family of $\alpha$-automorphisms:

\begin{prop}\label{family of alpha auto}
The following $f$ is an $\alpha$-endomorphism (automorphism):

$f(x)= c+ [(4b^2+1)/4b]x+ [(4b^2-1)/4b]y+ \sum_{j= 1,2,3,\ldots} a_{2j}\tilde{S}_{2j}$

$f(y)= c+ [(4b^2+1)/4b]y+ [(4b^2-1)/4b]x+ \sum_{j= 1,2,3,\ldots} a_{2j}\tilde{S}_{2j}$,

where \bit
\item $F \ni b \neq 0$. 
\item $F \ni c, a_2, a_4, a_6, \ldots$, with only finitely many nonzero scalars from this set.
\item $\tilde{S}_{2j}$ is symmetric and $(1,1)$ homogeneous of $(1, 1)$ degree $2j$ of the special following form:
$A_1 \ni \tilde{S}_{2j}= (X-Y)^{2j}$, with $(X-Y)^{2j}$ = the computation of $(X-Y)^{2j}$ in $F[X,Y]$ and then replacing $X,Y$ by $x,y$. 
\eit
\end{prop}

For example,
$f(x)= c+ [(4b^2+1)/4b]x+ [(4b^2-1)/4b]y+ a_2\tilde{S}_2 + a_4\tilde{S}_4$, 

$f(y)= c+ [(4b^2+1)/4b]y+ [(4b^2-1)/4b]x+ a_2\tilde{S}_2 + a_4\tilde{S}_4$, 

where $b \neq 0$, $\tilde{S}_2= (x^2+y^2)-2xy $, $\tilde{S}_4= (x^4+y^4)- 4(x^3y+xy^3)+ 6x^2y^2$.

\begin{proof}
Notice the following trivial fact:
Let $S$ be a symmetric element of $A_1$ and $K$ an antisymmetric element of $A_1$.
If $-(1/2)= [K,S]$, then $h(x)= S+K$, $h(y)= S-K$ is an $\alpha$-endomorphism of $A_1$.

Indeed, $h(x^*)= h(y)= S-K= (S+K)^*= (h(x))^*$ , $h(y^*)= h(x)= S+K= (S-K)^*= (h(y))^*$ and
$[h(y),h(x)]= [S-K,S+K]= [S,K]-[K,S]= -[K,S]-[K,S]= -2[K,S]= -2{-(1/2)}= 1$.

In view of the above trivial fact, in order to show that $f$ is an $\alpha$-endomorphism of $A_1$, it is enough to find a symmetric element $S$ and an antisymmetric element $K$ such that $-(1/2)= [K,S]$, 
and $f(x)= S+K$, $f(y)= S-K$.

Let $K= (1/4b)(x-y)$ and $S= c+ b(x+y)+ \sum_{j= 1,2,3,\ldots} a_{2j}\tilde{S}_{2j}$.
Of course, $K$ is antisymmetric and $S$ is symmetric (remember that the $\tilde{S}_{2j}$'s are symmetric).

$S+K= c+ b(x+y)+ \sum_{j= 1,2,3,\ldots} a_{2j}\tilde{S}_{2j}+ (1/4b)(x-y)=
c+ b(x+y)+ (1/4b)(x-y) + \sum_{j= 1,2,3,\ldots} a_{2j}\tilde{S}_{2j}=
c+ (4b^2/4b)(x+y)+ (1/4b)(x-y) + \sum_{j= 1,2,3,\ldots} a_{2j}\tilde{S}_{2j}=
c+ (4b^2+1/4b)x + (4b^2-1/4b)y + \sum_{j= 1,2,3,\ldots} a_{2j}\tilde{S}_{2j}=
f(x)$.

$S-K= c+ b(x+y)+ 
\sum_{j= 1,2,3,\ldots} a_{2j}\tilde{S}_{2j}- (1/4b)(x-y)=$
$c+ b(x+y)- (1/4b)(x-y) + \sum_{j= 1,2,3,\ldots} a_{2j}\tilde{S}_{2j}=$
$c+ (4b^2/4b)(x+y)- (1/4b)(x-y) + \sum_{j= 1,2,3,\ldots} a_{2j}\tilde{S}_{2j}=$
$c+ [(4b^2-1)/4b]x + [(4b^2+1)/4b]y + \sum_{j= 1,2,3,\ldots} a_{2j}\tilde{S}_{2j}=
f(y)$.

It remains to show that $-(1/2)= [K,S]$.

$[K,S]= [(1/4b)(x-y), c+ b(x+y)+ \sum_{j= 1,2,3,\ldots} a_{2j}\tilde{S}_{2j}]=$
$(1/4)[x-y,x+y] + (1/4b)[x-y,\sum_{j= 1,2,3,\ldots} a_{2j}\tilde{S}_{2j}]=$
$(1/4)([x,y]-[y,x]) + (1/4b)[x-y,\sum_{j= 1,2,3,\ldots} a_{2j}\tilde{S}_{2j}]=$
$(1/4){-1-1} + (1/4b)[x-y,\sum_{j= 1,2,3,\ldots} a_{2j}\tilde{S}_{2j}]=$
$-(1/2) + (1/4b)[x-y,\sum_{j= 1,2,3,\ldots} a_{2j}\tilde{S}_{2j}]$.

\textbf{Claim}: $[x-y, \sum_{j= 1,2,3,\ldots} a_{2j}\tilde{S}_{2j}]= 0$.
\textbf{Proof of claim}: It is well known that for every $w \in A_1$, 
$[x,w]= -(\partial / \partial y)(w)$ and $[y,w]= (\partial / \partial x)(w)$.
So, $[x-y, \sum_{j= 1,2,3,\ldots} a_{2j}\tilde{S}_{2j}]= 
[x, \sum_{j= 1,2,3,\ldots} a_{2j}\tilde{S}_{2j}] - [y, \sum_{j= 1,2,3,\ldots} a_{2j}\tilde{S}_{2j}]=$

$-(\partial / \partial y)(\sum_{j= 1,2,3,\ldots} a_{2j}\tilde{S}_{2j}) - 
(\partial / \partial x)(\sum_{j= 1,2,3,\ldots} a_{2j}\tilde{S}_{2j})=$

$- {(\partial / \partial y)(\sum_{j= 1,2,3,\ldots} a_{2j}\tilde{S}_{2j}) + 
(\partial / \partial x)(\sum_{j= 1,2,3,\ldots} a_{2j}\tilde{S}_{2j})}$.

It is not difficult to show that for each $j$: 
$(\partial / \partial y)(a_{2j}\tilde{S}_{2j}) + (\partial / \partial x)(a_{2j}\tilde{S}_{2j})= 0$, 
Therefore, 
$(\partial / \partial y)(\sum_{j= 1,2,3,\ldots} a_{2j}\tilde{S}_{2j}) + 
(\partial / \partial x)(\sum_{j= 1,2,3,\ldots} a_{2j}\tilde{S}_{2j})= 0$.

\textbf{For example}: 
$\tilde{S}_2= (x^2+y^2)-2xy$:
$(\partial / \partial y)(\tilde{S}_{2}) = (\partial / \partial y)(x^2+y^2-2xy)= 2y-2x$.
$(\partial / \partial x)(\tilde{S}_{2})= (\partial / \partial x)(x^2+y^2-2xy)= 2x-2y$.

$\tilde{S}_4= (x^4+y^4)- 4(x^3y+xy^3)+ 6x^2y^2$:

$(\partial / \partial y)(\tilde{S}_4) = $

$(\partial / \partial y)[x^4+y^4- 4(x^3y+xy^3)+ 6x^2y^2]=$

$4y^3-4x^3-12xy^2+12x^2y$.

$(\partial / \partial x)(\tilde{S}_4)= $

$(\partial / \partial x)[x^4+y^4- 4(x^3y+xy^3)+ 6x^2y^2]=$

$4x^3-12x^2y-4y^3+12xy^2$.

Another way which shows that $(1/2)= [S,K]$:

$\tilde{S}_2= x^2+y^2-2xy= (x-y)^2+1$ and 

$\tilde{S}_4= (x^4+y^4)-4(x^3y+xy^3)+6x^2y^2= (x-y)^4+4(x-y)^2+4$.

Then, $a_2\tilde{S}_2+a_4\tilde{S}_4=$

$a_2[(x-y)^2+1]+a_4[(x-y)^4+4(x-y)^2+4]=$

$a_4(x-y)^4+[a_2+4a_4](x-y)^2+[a_2+4a_4] \in F[x-y]$, 

hence 

$[a_2\tilde{S}_2+a_4\tilde{S}_4, x-y]= 0$.

Therefore, $[S,K]= 
[c+ b(x+y)+ a_2\tilde{S}_{2}+ a_4\tilde{S}_4, (1/4b)(x-y)]=$

$[b(x+y),(1/4b)(x-y)]+[a_2\tilde{S}_2+ a_4\tilde{S}_4, (1/4b)(x-y)]=$

$(1/4)[x+y,x-y]+ 0= (1/4){1+1}= (1/2)$.

Finally, one can show that a general $f$ from the family is onto (we will not show this). 
We just show that the following $f$ of degree (at most) $4$ is onto:
$f(x)= c+ [(4b^2+1)/4b]x+ [(4b^2-1)/4b]y+ a_2\tilde{S}_2 + a_4\tilde{S}_4$
$f(y)= c+ [(4b^2+1)/4b]y+ [(4b^2-1)/4b]x+ a_2\tilde{S}_2 + a_4\tilde{S}_4$
(where $b \neq 0$, $\tilde{S}_2= (x^2+y^2)-2xy $, $\tilde{S}_4= (x^4+y^4)- 4(x^3y+xy^3)+ 6x^2y^2$).

$f(A_1) \ni (f(x)-f(y))/2= ([(4b^2+1)/4b]x+ [(4b^2-1)/4b]y- [(4b^2+1)/4b]y- [(4b^2-1)/4b]x)/2=$
$([(4b^2+1)/4b- (4b^2-1)/4b]x + [(4b^2-1)/4b - (4b^2+1)/4b]y)/2=$
$([2/4b]x + [-2/4b]y)/2=$
$([1/2b]x + [-1/2b]y)/2=
(1/4b)(x-y)$.
 
So, $(x-y) \in f(A_1)$.

$f(A_1) \ni (x-y)(x-y)= x^2-xy-yx+y^2= x^2-xy-xy-1+y^2= x^2+y^2-2xy-1$ 

(of course, $(x-y)(x-y)$ is computed in $A_1$), hence (since $f(A_1) \ni 1$)

$f(A_1) \ni x^2+y^2-2xy$.

$f(A_1) \ni (2^{-1})(f(x)+f(y))= $
$(2^{-1})[c+ [(4b^2+1)/4b]x+ [(4b^2-1)/4b]y+ a_2\tilde{S}_2 + a_4\tilde{S}_4]+ $

$(2^{-1})[c+ [(4b^2+1)/4b]y+ [(4b^2-1)/4b]x+a_2\tilde{S}_2 + a_4\tilde{S}_4] =$

$c+ a_2\tilde{S}_2 + a_4\tilde{S}_4 + (1/2)[(4b^2+1)/4b+(4b^2-1)/4b]x +(1/2)[(4b^2-1)/4b+(4b^2+1)/4b]y=$

$c+ a_2\tilde{S}_2 + a_4\tilde{S}_4 + (1/2)(8b^2/4b)x +(1/2)(8b^2/4b)y=$

$c+ a_2\tilde{S}_2 + a_4\tilde{S}_4 + (4b^2/4b)x +(4b^2/4b)y=$

$c+ a_2\tilde{S}_2 + a_4\tilde{S}_4 + bx + by=$

$c+ a_2\tilde{S}_2 + a_4\tilde{S}_4 + b(x+y)$, 

so $f(A_1) \ni a_2\tilde{S}_2 + a_4\tilde{S}_4 + b(x+y)$.

But we have seen that $\tilde{S}_2= x^2+y^2-2xy \in f(A_1)$, 

therefore $f(A_1) \ni  a_4\tilde{S}_4 + b(x+y)$.

Now, a direct computation (in $A_1$) shows that 
$f(A_1) \ni (x^2+y^2-2xy)(x^2+y^2-2xy)=$
$(x^4+y^4)-4(x^3y+xy^3)+6x^2y^2-4(x^2+y^2-2xy)+2$.

Hence, $f(A_1) \ni (x^4+y^4)-4(x^3y+xy^3)+6x^2y^2= \tilde{S}_4$.

We have just seen that $f(A_1) \ni  a_4\tilde{S}_4 + b(x+y)$, so $f(A_1) \ni b(x+y)$.

Then $f(A_1) \ni (x+y)$ and we have seen that $f(A_1) \ni (x-y)$,
so $f(A_1) \ni x,y$ and $f$ is onto.

For a general $f$ from the family, similar computations show that $f(A_1) \ni x,y$.
\end{proof}

Some remarks:

\textbf{1.} Proposition \ref{family of alpha auto} shows that given any even number $n \geq 2$ (and also for $n=1$), there exists an $\alpha$-automorphism of degree $n$.
Indeed, for $F \ni b \neq 0$, $c,a_2,a_4,a_6, \ldots \in F$:
If $n=2J$ ($J \in \{1,2,3, \ldots\}$):
$f(x)= c+ [(4b^2+1)/4b]x+ [(4b^2-1)/4b]y+ 
\sum_{j= 1,2,3,\ldots,J} a_{2j}\tilde{S}_{2j}$.

$f(y)= c+ [(4b^2+1)/4b]y+ [(4b^2-1)/4b]x+ 
\sum_{j= 1,2,3,\ldots,J} a_{2j}\tilde{S}_{2j}$.

If $n=1$: 
$f(x)= c+ [(4b^2+1)/4b]x+ [(4b^2-1)/4b]y$.

$f(y)= c+ [(4b^2+1)/4b]y+ [(4b^2-1)/4b]x$.

This is the automorphism brought in Remark \ref{degree 1 is auto}: 

$f(x)= Ax+By+C, f(y)= Bx+Ay+C$, with $A^2- B^2= 1$.

Indeed, take $A= (4b^2+1)/4b$, $B= (4b^2-1)/4b$, $C=c$.

$A^2- B^2= [(4b^2+1)/4b]^2- [(4b^2-1)/4b]^2= $

$[(16b^4+8b^2+1)/16b^2]- [(16b^4-8b^2+1)/16b^2]= $

$16b^2/16b^2 =1$.

We wish to remark that we did somewhat tedious computations for a general $\alpha$-endomorphism $f$ of degree at most $m \in \{2,3,4\}$ and arrived at the following conclusions (which may or may not give a hint of what happens in higher degrees):
\bit
\item $m=2$: Every $\alpha$-endomorphism is of the following form:

$f(x)= c+ [(4b^2+1)/4b]x+ [(4b^2-1)/4b]y+ a_2(x^2+ y^2 -2xy)$, 

$f(y)= c+ [(4b^2+1)/4b]y+ [(4b^2-1)/4b]x+ a_2(x^2+ y^2 -2xy)$,

where $F \ni b \neq 0$, $c,a_2 \in F$.

(Of course, when $a_2 \neq 0$ the degree of $f$ is exactly $2$, and when $a_2= 0$ the degree of $f$ is exactly $1$).

\item $m=3$: There exists no $\alpha$-endomorphism of $A_1$ of degree $3$. 

{}From the general $f(x)= E_3+ E_2+ E_1+ E_0$ (where $E_i$ is $(1,1)$ homogeneous of degree $i$), 
we got (from solving $[f(x)^*,f(x)]= 1$) that $E_3=0$ and $f(x)= c+ [(4b^2+1)/4b]x+ [(4b^2-1)/4b]y+ a_2(x^2+ y^2 -2xy)$. Namely, we got a member of our family of degree at most $2$.

(Actually, we did a simpler computation, which relies on Theorem \ref{prime degree}).

\item $m=4$: Every $\alpha$-endomorphism is of the following form:

$f(x)= c+ [(4b^2+1)/4b]x+ [(4b^2-1)/4b]y+ a_2(x^2+ y^2 -2xy)+ a_4(x^4+ y^4 -4(x^3y+xy^3) -6x^2y^2$, 

$f(y)= c+ [(4b^2+1)/4b]y+ [(4b^2-1)/4b]x+ a_2(x^2+ y^2 -2xy)+ a_4(x^4+ y^4 -4(x^3y+xy^3) -6x^2y^2$,

where $F \ni b \neq 0$, $c,a_2,a_4 \in F$.

\eit

We will not bring our computations here, since they are just solving systems of equations (of degree $2$, since we have products of a coefficient from $f(x)$ with a coefficient from $f(y)$).

\textbf{2.} Let $S_{2j}$ be symmetric and $(1,1)$ homogeneous of $(1,1)$ degree $2j$ (apriori $S_{2j}$ may not equal the $a_{2j}\tilde{S}_{2j}$ of Proposition \ref{family of alpha auto}).

A general form of such an element is 
$S_{2j}= d_0(x^{2j}+y^{2j})+d_1(x^{2j-1}y+xy^{2j-1})+d_2(x^{2j-2}y^2+x^2y^{2j-2})+\ldots+
d_{j-1}(x^{2j-(j-1)}y^{j}+x^{j}y^{2j-(j-1)})+ d_j(x^jy^j)$.

Solving the following 
$(\partial / \partial y)(S_{2j}) + (\partial / \partial x)(S_{2j})= 0$, 

yields a unique solution, namely our $a_{2j}\tilde{S}_{2j}$ ($a_{2j} \in F$).

For example, $j=1$:
$S_{2}= d_0(x^2+y^2)+d_1xy$.
$(\partial / \partial y)(S_{2})= 2d_0y+d_1x$,
$(\partial / \partial x)(S_{2})= 2d_0x+d_1y$.
So, $(\partial / \partial y)(S_{2j}) + (\partial / \partial x)(S_{2j})= 
2d_0y+d_1x + {2d_0x+d_1y}=
(2d_0+d_1)y + (d_1+2d_0)x$.
Hence, $2d_0+d_1= 0$, so $d_1= -2d_0$ and we have 
$S_{2}= d_0(x^2+y^2) - 2d_0xy= d_0{(x^2+y^2)-2xy}$, as claimed.

\textbf{3.} The following family of functions is not a family of $\alpha$-endomorphisms (it is, only when all the $a_i$'s are zero),
it is not even a family of endomorphisms, since $[g(y),g(x)] \notin F$:

Let $F \ni b \neq 0$ and $F \ni c, a_3, a_5, a_7, \ldots$.
$g(x)= c+ [(4b^2+1)/4b]x+ [(4b^2-1)/4b]y+ \sum_{j= 1,2,3,\ldots} a_{2j + 1}\tilde{K}_{2j +1} $
$g(y)= c+ [(4b^2+1)/4b]y+ [(4b^2-1)/4b]x - \sum_{j= 1,2,3,\ldots} a_{2j + 1}\tilde{K}_{2j +1} $,

where only finitely many scalars from the set $F \ni c, a_3, a_5, a_7, \ldots$ are nonzero
and $\tilde{K}_{2j + 1}$ is an antisymmetric element of $(1, 1)$ degree $2j +1$ of the special following form:
$A_1 \ni \tilde{K}_{2j + 1}= (X-Y)^{2j +1}$, with $(X-Y)^{2j +1}$ = the computation of $(X-Y)^{2j +1}$ in $F[X,Y]$ and then replacing $X,Y$ by $x,y$.

In Proposition \ref{family of alpha auto} we had 
$(\partial / \partial y)(\tilde{S}_{2j}) + (\partial / \partial x)(\tilde{S}_{2j})= 0$. 

Here we must find antisymmetric $K_{2j +1}$ such that

$(\partial / \partial y)(K_{2j +1}) - (\partial / \partial x)(K_{2j +1})= 0$. 

However, a direct computation shows that there exists no such $K_{2j +1}$.

Indeed, let $K_{2j+ 1}=$
$d_0(x^{2j+1}-y^{2j+1})+d_1(x^{2j}y-xy^{2j})+d_2(x^{2j-1}y^2-x^2y^{2j-1})+\ldots+d_j(x^{2j+1-j}y^{j}-x^{j}y^{2j+1-j})$.

It is not difficult to solve 
$(\partial / \partial y)(K_{2j +1}) - (\partial / \partial x)(K_{2j +1})= 0$, 
and see that the only solution is $d_0= d_1= d_2= \ldots= d_j= 0$, hence $K_{2j +1}= 0$.
 
For example, for $j=1$, we have $K_3= d_0(x^3-y^3)+d_1(x^2y-xy^2)$, so
$(\partial / \partial y)(K_{3})= -3d_0y^2+ d_1x^2- 2d_1xy$ and
$(\partial / \partial x)(K_{3})= 3d_0x^2+ 2d_1xy- d_1y^2$.

Therefore,
$(\partial / \partial y)(K_{3}) - (\partial / \partial x)(K_{3})= $

$-3d_0y^2+ d_1x^2- 2d_1xy - {3d_0x^2+ 2d_1xy- d_1y^2}=$

$(-3d_0+d_1)y^2 + (d_1-3d_0)x^2 - 4d_1xy$.

So, $d_1-3d_0= 0$ and $-4d_1=0$, which implies that $3d_0= d_1= 0$, so $d_0= d_1= 0$.

For general $j$ we get $-2(j+1)d_jx^jy^j =0$, so $d_j=0$ and then all the other $d_i$'s are zero, since
$d_{j-1}$ is a multiple of $d_j$ hence $d_{j-1}=0$. $d_{j-2}$ is a multiple of $d_{j-1}$ etc.

As an exercise, one can check that 
$g(x)= c+ [(4b^2+1)/4b]x+ [(4b^2-1)/4b]y+ a_3\tilde{K}_3 + a_5\tilde{K}_5$, 

$g(y)= c+ [(4b^2+1)/4b]y+ [(4b^2-1)/4b]x- a_3\tilde{K}_3 - a_5\tilde{K}_5$,
 
where $b \neq 0$, $\tilde{K}_3= (x^3-y^3)-3(x^2y-xy^2)$, 

$\tilde{K}_5= (x^5-y^5)- 5(x^4y-xy^4)+ 10(x^3y^2-x^2y^3)$,
is not an ($\alpha$-)endomorphism.

Now we try to show that the starred Dixmier's conjecture is true, at least in some special cases.

\begin{thm}\label{starred true}
Let $f$ be an $\alpha$-endomorphism of $A_1$ of degree $n > 1$.

If $\rho(x^n+y^n)$ or $\rho(x^n-y^n)$ belongs to the $(1,1)$ leading term of $f(x)$ ($F \ni \rho \neq 0$),
then $f$ is onto.
\end{thm}

We quote Joseph's result \cite[Corollary 5.5]{joseph}: ``Let $f$ be an endomorphism of $A_1$. Then either $f$ is an
automorphism or there exists a positive integer $m$ and a map $\psi^{(m)} \in H^{(m)}$
such that for some $\Z \in l \neq 0, F \ni \beta \neq 0$, 
$l_{1, -1}(\psi^{(m)}(f(x))) = -(m/ \beta l) x^{- l/m}$ 
and $l_{1, -1}(\psi^{(m)}(f(y))) = \beta y x^{1+ l/m}$". 

Some remarks:\bit
\item We express each element $w \in A_1^{(m)}$ in the normal form $\sum \gamma_{ij} y^ix^{j/m}$, 
$\gamma_{ij} \in F$, $i \in \N$, $j \in \Z$.

\item $\psi^{(m)} \in H^{(m)}$ is an injective homomorphism of $A_1$ into $A_1^{(m)}$.

An accurate definition of $H^{(m)}$ can be found in \cite[page 605]{joseph}.

\item $\nu_{1, -1}(\psi^{(m)}(f(x)))= -(l/m)$ and $\nu_{1, -1}(\psi^{(m)}(f(y)))= l/m$. 

Therefore, we have $\nu_{1, -1}(\psi^{(m)}(f(x)))= -\nu_{1, -1}(\psi^{(m)}(f(y)))$. 
\eit

\begin{proof}
Let $f$ be an $\alpha$-endomorphism of $A_1$. 
We wish to show that $f$ is an automorphism of $A_1$.

By \cite[Corollary 5.5]{joseph}, it is enough to show that there exist no positive integer $m$ and a map $\psi^{(m)} \in H^{(m)}$
such that for some $\Z \ni l \neq 0, F \ni \beta \neq 0$, 
$l_{1, -1}(\psi^{(m)}(f(x))) = -(m/ \beta l) x^{- l/m}$ 
and $l_{1, -1}(\psi^{(m)}(f(y))) = \beta y x^{1+ l/m}$.

Otherwise, let $m$ be a positive integer and $\psi^{(m)} \in H^{(m)}$ a map
such that for some $\Z \ni l \neq 0, F \ni \beta \neq 0$, 
$l_{1, -1}(\psi^{(m)}(f(x))) = -(m/ \beta l) x^{- l/m}$ 
and $l_{1, -1}(\psi^{(m)}(f(y))) = \beta y x^{1+ l/m}$.

Denote: $\psi^{(m)}(x)= A= A_l+ \epsilon$, $\psi^{(m)}(y)= B= B_l+ \delta$,
where $A_l= l_{1,-1}(A)$, $B_l= l_{1,-1}(B)$.

($A, B, A_l, B_l, \epsilon, \delta \in A_1^{(m)}$).

$\nu_{1, -1}(A)= a \in \Q$, $\nu_{1, -1}(B)= b \in \Q$.

(Hence, $\nu_{1, -1}(A_l)= a$, $\nu_{1, -1}(B_l)= b$).

$f(x)= E_n+ E_{n-1}+ \ldots + E_1+ E_0$, where $n > 1$ 
(we have remarked in Remark \ref{degree 1 is auto} that an $\alpha$-endomorphism of degree $1$ is an automorphism), 

$E_n \neq 0$ and $E_i$ ($0 \leq i \leq n$) is $(1, 1)$ homogeneous of $(1, 1)$ degree $i$. 

$E_i= S_i+ K_i$ where $S_i= [E_i+(E_i)^*]/2$ and $K_i= [E_i-(E_i)^*]/2$ ($0 \leq i \leq n$).

(Notice that $K_0= 0$).

$S= S_n+S_{n-1}+\ldots+S_1+S_0$ and $K= K_n+K_{n-1}+\ldots+K_1+K_0$.

Hence, $f(x)= E_n+ E_{n-1}+ \ldots + E_1+ E_0= S_n+K_n+S_{n-1}+K_{n-1}+\ldots+S_1+K_1+S_0+K_0= S+K$,
and $f(y)= f(x)^*= S-K$.

Since $\nu_{1, -1}(\psi^{(m)}(f(x)))= -\nu_{1, -1}(\psi^{(m)}(f(y)))$ (and each is nonzero), 
we may assume w.l.o.g that 
$\Q \ni \nu_{1, -1}(\psi^{(m)}(f(x)))$ is positive, denote it by $q > 0$.
Then, $\Q \ni \nu_{1, -1}(\psi^{(m)}(f(y)))= -q < 0$.

In those notations we have:

$0 < q= \nu_{1, -1}(\psi^{(m)}(f(x)))= \nu_{1, -1}(\psi^{(m)}(S+K))=$
$\nu_{1, -1}(\psi^{(m)}(S) + \psi^{(m)}(K))$ and

$0 > -q= \nu_{1, -1}(\psi^{(m)}(f(y)))= \nu_{1, -1}(\psi^{(m)}(S-K))= $
$\nu_{1, -1}(\psi^{(m)}(S) - \psi^{(m)}(K))$.

There are (w.l.o.g) three cases:\bit
\item $b \leq a \leq 0$.
\item $b \leq 0 \leq a$.
\item $0 \leq b \leq a$. Here there are two options: $0 \leq b < a$ and $0 \leq b= a$.
\eit
Now we show that the first two cases and the first option of the third case are impossible.

\textbf{First case} $b \leq a \leq 0$: 
Clearly (even without knowing that $\rho(x^n+y^n)$ or $\rho(x^n-y^n)$ belongs to the $(1,1)$ leading term of $f(x)$)
it is impossible to have $\nu_{1, -1}(\psi^{(m)}(f(x)))= q > 0$.

Indeed, $(\psi^{(m)}(f(x)))$ is a polynomial in $A$ and $B$, 
with $\nu_{1, -1}(A)= a \leq 0$ and $\nu_{1, -1}(B)= b \leq 0$.
Therefore, Proposition \ref{degree} implies that $(\psi^{(m)}(f(x))) \leq 0$.

\textbf{Second case} $b \leq 0 \leq a$: 
{}From Lemma \ref{useful lemma}, we get that the $(1,1)$ leading term of $f(x)$, $E_n$, is symmetric or antisymmetric.

\bit
\item If $l_{1, 1}(f(x))= E_n$ is symmetric ($E_n=S_n$), then $\rho(x^n+y^n)$ belongs to $E_n$.
\item If $l_{1, 1}(f(x))= E_n$ is antisymmetric ($E_n=K_n$), then $\rho(x^n-y^n)$ belongs to $E_n$.
\eit
We only show what happens if $l_{1, 1}(f(x))= E_n$ is symmetric and $\rho(x^n+y^n)$ belongs to $E_n$.
(If $l_{1, 1}(f(x))= E_n$ is antisymmetric and $\rho(x^n-y^n)$ belongs to $E_n$, we get similar results).

Write $E_n= S_n= \gamma(x^n+y^n)+ D_n$, where $F \in \gamma \neq 0$ and $D_n$ is symmetric and $(1,1)$ homogeneous of $(1,1)$ degree $n$ in which (a nonzero scalar multiple of) $x^n+y^n$ not appears.

Then,
$\psi^{(m)}(f(y))= $
$\psi^{(m)}((E_n)^*+ (E_{n-1})^*+ \ldots + (E_1)^*+ (E_0)^*)=$

$\psi^{(m)}(S_n+ (E_{n-1})^*+ \ldots + (E_1)^*+ (E_0)^*)=$

$\psi^{(m)}(\gamma(x^n+y^n)+ D_n+ (E_{n-1})^*+ \ldots + (E_1)^*+ (E_0)^*)=$

$\gamma(A^n+B^n)+ \psi^{(m)}(D_n)+ \psi^{(m)}((E_{n-1})^*)+ \ldots + \psi^{(m)}((E_1)^*)+ \psi^{(m)}((E_0)^*)$.

And $\psi^{(m)}(f(x))=$

$\gamma(A^n+B^n)+ \psi^{(m)}(D_n)+ \psi^{(m)}(E_{n-1})+ \ldots + \psi^{(m)}(E_1)+ \psi^{(m)}(E_0)$.
Notice the following trivial claim:
Let $0 \leq a \in \Q$, $0 \geq b \in \Q$, $0 \leq \alpha \in \N$, $0 \leq \beta \in \N$, 
$0 \leq n \in \N$ such that $\alpha+ \beta \leq n$.  
Then, $na \geq \alpha a+ \beta b$.

Proof of claim: $na \geq \alpha a+ \beta b$ $\Leftrightarrow$ $(n-\alpha)a \geq \beta b$.
Since $0 \leq \beta \leq (n-\alpha)$ 
and $b \leq 0 \leq a$, we get $\beta b \leq (n-\alpha)a$, since $(n-\alpha)a$ is non-negative, while $\beta b$ is non-positive. 

Let $0 \leq \alpha \in \N$ and $0 \leq \beta \in \N$ such that $\alpha+ \beta \leq n$ 

($n$ is, of course, the degree of $f$). 

Then from the above trivial claim, $na \geq \alpha a+ \beta b$, where $a$ and $b$ are, of course, $\nu_{1, -1}(A_l)= a$, $\nu_{1, -1}(B_l)= b$.
We now show that: \bit 
\item $\nu_{1, -1}(\psi^{(m)}(f(y)))= na \geq 0$ - a contradiction to $\nu_{1, -1}(\psi^{(m)}(f(y))) < 0$. 
Or
\item $\nu_{1, -1}(\psi^{(m)}(f(x))) \leq 0$ - a contradiction to $\nu_{1, -1}(\psi^{(m)}(f(x))) > 0$.
\eit

If $na > \alpha a+ \beta b$, then $l_{1,-1}(\psi^{(m)}(f(y)))= \gamma (A_l)^n$.

Therefore, $\nu_{1, -1}(\psi^{(m)}(f(y)))= na \geq 0$, 

a contradiction to $\nu_{1, -1}(\psi^{(m)}(f(y))) < 0$.

If $na = \alpha a+ \beta b$, then: \bit
\item If $a=b=0$, then $\nu_{1, -1}(\psi^{(m)}(f(x))) \leq 0$, a contradiction to 
$\nu_{1, -1}(\psi^{(m)}(f(x))) > 0$.

\item If $b < 0=a$ ($\beta=0$), then $\nu_{1, -1}(\psi^{(m)}(f(x))) \leq 0$, 
a contradiction to $\nu_{1, -1}(\psi^{(m)}(f(x))) > 0$.

\item If $b=0< a$ ($\alpha= n$), then $l_{1,-1}(\psi^{(m)}(f(y)))= \gamma (A_l)^n$. 

Therefore, $\nu_{1, -1}(\psi^{(m)}(f(y)))= na > 0$, a contradiction to $\nu_{1, -1}(\psi^{(m)}(f(y))) < 0$.

\item If $b < 0 < a$ ($\alpha =n$ and $\beta=0$), then $l_{1,-1}(\psi^{(m)}(f(y)))= \gamma (A_l)^n$. 

Therefore, $\nu_{1, -1}(\psi^{(m)}(f(y)))= na > 0$, a contradiction to $\nu_{1, -1}(\psi^{(m)}(f(y))) < 0$.
\eit

\textbf{Third case, first option} $0 \leq b < a$: It is similar to the second case above. 
{}From Lemma \ref{useful lemma}, we get that the $(1,1)$ leading term of $f(x)$, $E_n$, is symmetric or antisymmetric.

\bit
\item If $l_{1, 1}(f(x))= E_n$ is symmetric ($E_n=S_n$), then $\rho(x^n+y^n)$ belongs to $E_n$.
\item If $l_{1, 1}(f(x))= E_n$ is antisymmetric ($E_n=K_n$), then $\rho(x^n-y^n)$ belongs to $E_n$.
\eit
We only show what happens if $l_{1, 1}(f(x))= E_n$ is antisymmetric and $\rho(x^n-y^n)$ belongs to $E_n$.
(If $l_{1, 1}(f(x))= E_n$ is symmetric and $\rho(x^n+y^n)$ belongs to $E_n$, we get similar results).

Write $E_n= K_n= \gamma(x^n-y^n)+ \tilde{D_n}$, where $F \in \gamma \neq 0$ and $\tilde{D_n}$ is antisymmetric, $(1,1)$ homogeneous of $(1,1)$ degree $n$, in which (a nonzero scalar multiple of) $x^n-y^n$ not appears.

Then, 
$\psi^{(m)}(f(y))=$

$\psi^{(m)}(-K_n+ (E_{n-1})^*+ \ldots + (E_1)^*+ (E_0)^*)=$

$\psi^{(m)}(-\gamma(x^n-y^n)- \tilde{D_n}+ (E_{n-1})^*+ \ldots + (E_1)^*+ (E_0)^*)=$

$-\gamma(A^n-B^n)- \psi^{(m)}(\tilde{D_n})+ \psi^{(m)}((E_{n-1})^*)+ \ldots + \psi^{(m)}((E_1)^*)+ \psi^{(m)}((E_0)^*)$.
Again we have a trivial claim:
Let $a \in \Q$, $b \in \Q$, $0 \leq b < a$, $0 \leq \alpha \in \N$, $0 \leq \beta \in \N$, 
$0 \leq n \in \N$ such that $\alpha+ \beta \leq n$.  
Then, $na > \alpha a+ \beta b$.

Proof of claim: $na > \alpha a+ \beta b$ $\Leftrightarrow$ $(n-\alpha)a > \beta b$.
Since $0 \leq \beta \leq (n-\alpha)$ and $0 \leq b < a$, we get $\beta b < (n-\alpha)a$.
Let $0 \leq \alpha \in \N$ and $0 \leq \beta \in \N$ such that $\alpha+ \beta \leq n$ ($n$ is, of course, the degree of $f$). Then from the above trivial claim, $na > \alpha a+ \beta b$, where $a$ and $b$ are, of course, $\nu_{1, -1}(A_l)= a$, $\nu_{1, -1}(B_l)= b$.
Therefore, $l_{1,-1}(\psi^{(m)}(f(y)))= -\gamma (A_l)^n$ and $\nu_{1, -1}(\psi^{(m)}(f(y)))= na > 0$, 

a contradiction to $\nu_{1, -1}(\psi^{(m)}(f(y))) < 0$.

It remains to show that the second option of the third case is impossible.
It is a lot complicated, and we hope that it is indeed impossible as all the other cases are impossible. 
If it is possible ($\psi^{(m)}$ exists), then it may help us in finding (although not so quickly) a counterexample, namely, an $\alpha$-endomorphism which is not onto.

\textbf{Third case, second option} $0 \leq b= a$ (Unfortunately, it is not yet fully understood, so our theorem may not be proved): Of course, when $b=0$ we get $a=b=0$ which we dealt with in the first case. Hence we assume that $0 < b=a$.

We try to show that there exists no such $\psi^{(m)}$.

As we have already seen above, one can write $E_n= S_n= \gamma(x^n+y^n)+D_n$ or $E_n= K_n= \gamma(x^n-y^n)+\tilde{D_n}$, where $F \in \gamma \neq 0$, $D_n$ is symmetric $(1,1)$ homogeneous of $(1,1)$ degree $n$ in which (a nonzero scalar multiple of) $x^n+y^n$ not appears and $\tilde{D_n}$ is antisymmetric $(1,1)$ homogeneous of $(1,1)$ degree $n$ in which (a nonzero scalar multiple of) $x^n-y^n$ not appears.

We divide to two options: $A_l \neq B_l$ and $A_l = B_l$.

$A_l \neq B_l$ ($\epsilon$ may or may not equal $\delta$): 
Assume that $E_n= K_n= \gamma(x^n-y^n)+\tilde{D_n}$ ($E_n= S_n= \gamma(x^n+y^n)+D_n$ yields similar results).

Then, $\psi^{(m)}(f(y))=$

$\psi^{(m)}(-K_n+ (E_{n-1})^*+ \ldots + (E_1)^*+ (E_0)^*)=$

$\psi^{(m)}(-\gamma(x^n-y^n)- \tilde{D_n}+ (E_{n-1})^*+ \ldots + (E_1)^*+ (E_0)^*)=$

$-\gamma(A^n-B^n)- \psi^{(m)}(\tilde{D_n})+ \psi^{(m)}((E_{n-1})^*)+ \ldots + $

$\psi^{(m)}((E_1)^*)+ \psi^{(m)}((E_0)^*)=$

$-\gamma[(A_l+\epsilon)^n-(B_l+\delta)^n]- \psi^{(m)}(\tilde{D_n})+ \psi^{(m)}((E_{n-1})^*)+ \ldots +$

$\psi^{(m)}((E_1)^*)+ \psi^{(m)}((E_0)^*)$.

Write $\tilde{D_n}= \sum \gamma_{ij}x^iy^j$ and let $T_n= \sum \gamma_{ij}(A_l)^i(B_l)^j$.

Then the $(1,1)$ leading term of $\psi^{(m)}(f(y))$ is 

$-\gamma[(A_l)^n-(B_l)^n]- T_n$, 

since it has $(1,1)$ degree $na$ and all the other terms appearing in $\psi^{(m)}(f(y))$ are of degrees $ < na$.

Therefore, $\nu_{1, -1}(\psi^{(m)}(f(y)))= na > 0$, a contradiction to $\nu_{1, -1}(\psi^{(m)}(f(y))) < 0$.

Unless,
$-\gamma[(A_l)^n-(B_l)^n]- T_n = 0$.

We suspect that the relation $-\gamma[(A_l)^n-(B_l)^n]- T_n = 0$ can not hold in $A_1^{(m)}$.
(Maybe this follows directly from the definition of $A_1^{(m)}$).


$A_l = B_l$: 
Necessarily $\epsilon \neq \delta$. Indeed, if $\epsilon= \delta$, 
then $A= A_l+ \epsilon$ and $B= B_l+ \delta= A_l+ \epsilon$, 
hence $1= [\psi^{(m)}(y), \psi^{(m)}(x)]= [B,A]= [A_l+ \epsilon, A_l+ \epsilon]= 0$.

We are not able to show that $\nu_{1, -1}(\psi^{(m)}(f(y))) > 0$ or that 
$\nu_{1, -1}(\psi^{(m)}(f(x))) < 0$.

Maybe in order to show that there exists no such $\psi^{(m)}$, one may use the following claim:
\textbf{Claim}: Each $(1,-1)$ non-negative component in $\psi^{(m)}(S)$
must also appear (as it is, not multiplied by a scalar $\neq 1$) in $\psi^{(m)}(K)$.  
(In other words, the $(1,-1)$ non-negative components of $\psi^{(m)}(S)$ and $\psi^{(m)}(K)$ are the same).

\textbf{Proof of claim}: Otherwise, there exists a $(1,-1)$ non-negative component in $\psi^{(m)}(S)$
which not appears as it is in $\psi^{(m)}(K)$ 

(this $(1, -1)$ non-negative component appears in $\psi^{(m)}(K)$ multiplied by scalar $\neq 1$).
But then, 
$0 > -q= \nu_{1, -1}(\psi^{(m)}(f(y)))= \nu_{1, -1}(\psi^{(m)}(S) - \psi^{(m)}(K)) \geq 0$, 

since this $(1,-1)$ non-negative component (multiplied by some scalar $\neq 0$) appears in $\psi^{(m)}(S) - \psi^{(m)}(K)$.

(Remember that $f(y)= S-K$).

We do not know how to show that there exists a $(1,-1)$ non-negative component in $\psi^{(m)}(S)$ 
which not appears in $\psi^{(m)}(K)$.

However, it may happen that all the $(1,-1)$ non-negative components of $\psi^{(m)}(S)$ and of $\psi^{(m)}(K)$ are the same, $\nu_{1, -1}(\psi^{(m)}(f(x))) > 0$ and $\nu_{1, -1}(\psi^{(m)}(f(y))) < 0$, but still there is no such $\psi^{(m)} \in H^{(m)}$. 

Maybe even in this unfortunate case, it is still impossible to have 
$l_{1, -1}(\psi^{(m)}(f(x))) = -(m/ \beta l) x^{- l/m}$ 

and $l_{1, -1}(\psi^{(m)}(f(y))) = \beta y x^{1+ l/m}$ 

($m$ positive integer, $\Z \ni l \neq 0$, $F \ni \beta \neq 0$),

where $f$ is our given $\alpha$-endomorphism.

Unfortunately, we are not able to show that it is indeed impossible. 

\end{proof}

The condition that $\rho(x^n+y^n)$ or $\rho(x^n-y^n)$ belongs to the leading $(1, 1)$ term of $f(x)$
($f$ an $\alpha$-endomorphism of $A_1$ of degree $n > 1$, $F \ni \rho \neq 0$) seems reasonable in view of:\bit
\item Proposition \ref{family of alpha auto}: Obviously, every member of the family of degree $\geq 2$ satisfies this condition.
\item Theorem \ref{prime degree}, which will be brought soon, which shows that if $f$ is an $\alpha$-endomorphism of prime degree $p > 2$, then $\rho(x^p-y^p)$ belongs to the $(1,1)$ leading term of $f(x)$
(since the $(1,1)$ leading term of $f(x)$ is actually $\lambda(X-Y)^p$, $F \ni \lambda \neq 0$).
\eit

The following trivial lemma is needed for the proof of Theorem \ref{prime degree}.

\begin{lem}\label{trivial lemma}
Let $w \in A_1$ be $(1,1)$ homogeneous of prime degree $p > 2$.
\bit
\item If $w$ is symmetric and there exists $aX+ bY \in F[X,Y]$ such that $w= (aX+ bY)^p$, then $a= b$.

\item If $w$ is antisymmetric and there exists $aX+ bY \in F[X,Y]$ such that $w= (aX+ bY)^p$, then $a= -b$.

\eit
\end{lem}

Again, when we write $w= (aX+ bY)^p$, we mean the computation of $(aX+ bY)^p$ in $F[X,Y]$ and then replacing $X,Y$ by $x,y$.

\begin{proof}
\bit
\item $w= a^px^p+ pa^{p-1}x^{p-1}by+ p(p-1)/2 a^{p-2}x^{p-2}b^2y^2+ \ldots +$

$p(p-1)/2 a^2x^2b^{p-2}y^{p-2}+ paxb^{p-1}y^{p-1}+ b^py^p=$

$a^px^p+ pa^{p-1}bx^{p-1}y+ p(p-1)/2 a^{p-2}b^2x^{p-2}y^2+ \ldots +$

$p(p-1)/2 a^2b^{p-2}x^2y^{p-2}+ pab^{p-1}xy^{p-1}+ b^py^p$.

Then, $w^*= a^px^p+ pa^{p-1}bxy^{p-1}+ p(p-1)/2 a^{p-2}b^2x^2y^{p-2}+ \ldots +$

$p(p-1)/2 a^2b^{p-2}x^{p-2}y^2+ pab^{p-1}x^{p-1}y+ b^px^p$.

$w$ is symmetric, so $w= w^*$: 
$a^px^p+ pa^{p-1}bx^{p-1}y+ p(p-1)/2 a^{p-2}b^2x^{p-2}y^2+ \ldots +$

$p(p-1)/2 a^2b^{p-2}x^2y^{p-2}+ pab^{p-1}xy^{p-1}+ b^py^p=$ 

$a^px^p+ pa^{p-1}bxy^{p-1}+ p(p-1)/2 a^{p-2}b^2x^2y^{p-2}+ \ldots +$

$p(p-1)/2 a^2b^{p-2}x^{p-2}y^2+ pab^{p-1}x^{p-1}y+ b^px^p$, which implies that:

$a^p= b^p$, $pa^{p-1}b= pab^{p-1}$, $p(p-1)/2 a^{p-2}b^2= p(p-1)/2 a^2b^{p-2}$, etc.

Necessarily $a \neq 0$ and $b \neq 0$. Otherwise, if $a=0$, then from $a^p= b^p$ we get $b=0$. 
But then $w= (aX+ bY)^p= 0$, a contradiction to our assumption that $w$ is of prime degree $p >2$ 
(hence, in particular, $w \neq 0$).

So, $a \neq 0$ and $b \neq 0$. Therefore: $pa^{p-1}b= pab^{p-1}$ implies $a^{p-2}= b^{p-2}$.
$p(p-1)/2 a^{p-2}b^2= p(p-1)/2 a^2b^{p-2}$ implies $a^{p-4}= b^{p-4}$.

Continuing in this way, until $(p (p-1)/2) a^{p-(p-1)/2}b^{p-(p-1)/2-1}=$

$(p (p-1)/2) a^{p-(p-1)/2-1}b^{p-(p-1)/2}$, so 

$(p (p-1)/2) a^{(p+1)/2}b^{(p-1)/2}= (p (p-1)/2) a^{(p-1)/2}b^{(p+1)/2}$.

This last equality implies that $a= b$ 

(we divided by $(p (p-1)/2) a^{(p-1)/2} b^{(p-1)/2}$).

\item The only difference between the antisymmetric case and the symmetric case is in sign. 
Now we have $a^p= -b^p$, $pa^{p-1}b= -pab^{p-1}$, $p(p-1)/2 a^{p-2}b^2= -p(p-1)/2 a^2b^{p-2}$, etc.
Again, $a \neq 0$ and $b \neq 0$. 
If we divide the last equality $(p (p-1)/2) a^{(p+1)/2}b^{(p-1)/2}= -(p (p-1)/2) a^{(p-1)/2}b^{(p+1)/2}$
by 
$(p (p-1)/2) a^{(p-1)/2}b^{(p-1)/2}$, we get $a= -b$.

\eit
\end{proof}

\begin{thm}\label{prime degree}
Let $f$ be an $\alpha$-endomorphism of $A_1$ of prime degree $p > 2$. Let $f(x)= P$ (hence $f(y)= P^*$). 
Then the $(1, 1)$ leading term of $P$, $l_{1, 1}(P)$ is antisymmetric.

Moreover, $l_{1, 1}(P)= \lambda(X-Y)^p$ where $F \ni \lambda \neq 0$ and by $(X-Y)^p$ we mean the computation of $(X-Y)^p$ in $F[X,Y]$ and then replacing $X,Y$ by $x,y$.
\end{thm}

We have remarked above (after Proposition \ref{family of alpha auto}) that there exists no $\alpha$-endomorphism of $A_1$ of degree exactly $3$.
We hope to check if there exists an $\alpha$-endomorphism of $A_1$ of degree $5$. However, it may not give a clue of what happens in higher prime degrees. 

\begin{remark}
For $f$ an $\alpha$-endomorphism of $A_1$ of prime degree $p = 2$, we have already mentioned that a direct computation shows that $f$ must be of the following form: 
$f(x)= c+ [(4b^2+1)/4b]x+ [(4b^2-1)/4b]y+ a(x^2+ y^2 -2xy)$, 
$f(y)= c+ [(4b^2+1)/4b]y+ [(4b^2-1)/4b]x+ a(x^2+ y^2 -2xy)$.

So, the $(1,1)$ leading term of $f(x)$, $l_{1, 1}(f(x))$, is symmetric.
\end{remark}

\begin{proof}
{}From Lemma \ref{useful lemma}, the $(1, 1)$ leading term of $P$, $l_{1, 1}(P)$, is symmetric or antisymmetric.
Hence, we must show that $l_{1, 1}(P)$ is not symmetric.

Otherwise, $l_{1, 1}(P)$ is symmetric.

We can write $P= E_p+ E_{p-1}+ \ldots + E_1+ E_0$, where $E_p \neq 0$ and $E_i$ ($0 \leq i \leq p$) is $(1, 1)$ homogeneous of $(1, 1)$ degree $i$. Our assumption is that 
$l_{1, 1}(P)= E_p$ is symmetric (so $(E_p)^*= E_p$).

Since $f$ is an $\alpha$-endomorphism of $A_1$, 
$1= [f(y), f(x)]= [P^*, P]= [E_p+ (E_{p-1})^*+ \ldots + (E_1)^*+ E_0, E_p+ E_{p-1}+ \ldots + E_1+ E_0]=$
$[E_p, E_p]+ [E_p, E_{p-1}+ \ldots + E_1+ E_0]+ [(E_{p-1})^*+ \ldots + (E_1)^*+ E_0, E_p]+ 
[(E_{p-1})^*+ \ldots + (E_1)^*+ E_0, E_{p-1}+ \ldots + E_1+ E_0]$
($E_0 \in F$, so $(E_0)^*= E_0$).

Let $E= E_{p-1}+ \ldots + E_1+ E_0$, so $E^*= (E_{p-1})^*+ \ldots + (E_1)^*+ E_0$.
So, $1= [E_p, E]+ [E^*, E_p]+ [E^*, E]= [E_p, E]+ [E_p, -E^*]+ [E^*, E]=$
$[E_p, E-E^*]+ [E^*, E]= [E_p, E-E^*]+ [E^*, E-E^*]= [E_p+ E^*, E-E^*]= [P^*, E-E^*]$.

Therefore, we have $1= [P^*, E-E^*]$.

Let $m$ be maximal among $i \in \{0,1,\ldots,p-2,p-1\}$ such that $E_i-(E_i)^* \neq 0$ 
($m \geq 1$, since $E_0-(E_0)^* = 0$).

Just for convenience, apply $\alpha$ on both sides and get $1= [(E-E^*)^*, (P^*)^*]= [E^*-E, P]= [P, E-E^*]$.

We can apply \cite[Theorem 1.22(2)]{shape of possible counterexamples} to $1= [P, E-E^*]$.

Indeed, $[P, E-E^*]_{1, 1}= 0$, 
since $\nu_{1, 1}([P, E-E^*])= \nu_{1, 1}(1)= 0 < \nu_{1, 1}(P) +\nu_{1, 1}(E-E^*) - 2$ 
(we assumed that $f$ is of degree $\geq 3$, hence $\nu_{1, 1}(P) +\nu_{1, 1}(E-E^*) - 2 \geq 3+ 0- 2= 1$).

Therefore, there exist $0 \neq \lambda \in F$ and $0 \neq \mu \in F$, 

$M,N \in \N$ with $\gcd(M,N)= 1$ and a $(1, 1)$ homogeneous polynomial $R \in F[X,Y]$, such that:

$M/N= \nu_{1, 1}(P) / \nu_{1, 1}(E-E^*)$, the $(1, 1)$ leading term of $P$ is $\lambda R^M$ and 

the $(1, 1)$ leading term of $E-E^*$ is $\mu R^N$. 

{}From $\nu_{1, 1}(P)= p$ and $\nu_{1, 1}(E-E^*)= m$ ($1 \leq m \leq p-1$), 
we get $M/N= \nu_{1, 1}(P) / \nu_{1, 1}(E-E^*)= p/m$.
It is easy to see that $1= \gcd(M,N)$, implies that $M=p$ and $N=m$.
Hence, the $(1, 1)$ leading term of $P$ is $\lambda R^p$ and the $(1, 1)$ leading term of $E-E^*$ is $\mu R^m$. 

Of course, the $(1, 1)$ leading term of $P$ is $E_p$ and the $(1, 1)$ leading term of $E-E^*$ is $E_m-(E_m)^*$. 
Therefore, $E_p= \lambda R^p$ and $E_m-(E_m)^*= \mu R^m$.

Hence we have:\bit
\item $E_p= \lambda R^p$ symmetric ($E_p$ is symmetric by assumption).
\item $E_m-(E_m)^*= \mu R^m$ antisymmetric ($E_m-(E_m)^*$ is, of course, antisymmetric).
\eit

$E_p$ is $(1,1)$ homogeneous of $(1,1)$ degree $p$, so (from $E_p= \lambda R^p$) $R$ must be of degree $1$. $R$ homogeneous of degree $1$ is necessarily of the form $aX+bY$ ($a,b \in F$). Therefore, we have $E_p= \lambda R^p= \lambda (aX+bY)^p$.
Apply Lemma \ref{trivial lemma} to the symmetric $E_p= \lambda (aX+bY)^p$, and get $a=b$, hence $R= a(X+Y)$.
($E_p= \lambda a^p(X+Y)^p$). But then the antisymmetric $E_m-(E_m)^*= \mu R^m= \mu (a(X+Y))^m= \mu a^m (X+Y)^m$, 
which is impossible, since $(X+Y)^m$ is symmetric. Concluding that $l_{1, 1}(P)$ must be antisymmetric.

Next we show that $l_{1, 1}(P)= \lambda(X-Y)^p$: 

Write $f(x)= E_P+E_{p-1}+E_{p-2}+\ldots+E_1+E_0$, where $E_i$ is $(1,1)$ homogeneous of $(1,1)$ degree $i$.
Write $E_i= S_i+K_i$ with $S_i$ symmetric and $(1,1)$ homogeneous of $(1,1)$ degree $i$, and $K_i$ antisymmetric and $(1,1)$ homogeneous of $(1,1)$ degree $i$.
We have just seen that $E_p= K_p$ ($S_p= 0$).

Now, $-1= [f(x),f(x)^*]= [K_p+E_{p-1}+E_{p-2}+\ldots+E_1+E_0, -K_p+(E_{p-1}+E_{p-2}+\ldots+E_1+E_0)^*]=$
$[K_p,(E_{p-1}+E_{p-2}+\ldots+E_1+E_0)^*]+[E_{p-1}+E_{p-2}+\ldots+E_1+E_0, -K_p]+ 
[E_{p-1}+E_{p-2}+\ldots+E_1+E_0, (E_{p-1}+E_{p-2}+\ldots+E_1+E_0)^*]$. 

Let $E= E_{p-1}+ \ldots + E_1+ E_0$, so $E^*= (E_{p-1})^*+ \ldots + (E_1)^*+ E_0$.
So, $-1= [K_p, E^*]+[K_p, E]+ [E, E^*]= [K_p, E^*+E]+ [E, E^*+E]= [k_p+E, E^*+E]= [f(x), E^*+E]$.

Now, $l_{1, 1}(f(x))= K_p$ and $l_{1, 1}(E^*+E)=(E_m)^*+E_m$,
where $m$ is maximal among $i \in \{p-1, p-2, \ldots, 1, 0\}$ such that $(E_i)^*+E_i \neq 0$.

We can apply \cite[Theorem 1.22(2)]{shape of possible counterexamples} to $-1= [f(x), E^*+E]$.

Therefore, there exist $0 \neq \lambda \in F$ and $0 \neq \mu \in F$, 

$M,N \in \N$ with $\gcd(M,N)= 1$ and a $(1, 1)$ homogeneous polynomial $R \in F[X,Y]$, such that:

$M/N= \nu_{1, 1}(f(x)) / [\nu_{1, 1}(E^*+E)$, 

the $(1, 1)$ leading term of $f(x)$ is $\lambda R^M$ and the $(1, 1)$ leading term of $E^*+E$ is $\mu R^N$.

Hence, $K_p= \lambda R^M$ and $(E_m)^*+E_m= \mu R^N$.
 
{}From $\nu_{1, 1}(f(x))= \nu_{1,1}(K_p)= p$ and 

$\nu_{1, 1}(E^*+E)= m$, 

we get 

$M/N= [\nu_{1, 1}(f(x)) / \nu_{1, 1}(E^*+E)]= p/m$.

But $1= \gcd(M,N)$, so $M=p$ and $N=m$.

Hence, $K_p= \lambda R^p$ and $(E_m)^*+E_m= \mu R^m$. 

$K_p$ is homogeneous of $(1,1)$ degree $p$, so (from $K_p= \lambda R^p$) $R$ must be of degree $1$.
$R$ homogeneous of degree $1$ is necessarily of the form $aX+bY$ ($a,b \in F$, with at least one of $a,b$ nonzero).
Therefore, we have $K_p= \lambda R^p= \lambda (aX+bY)^p$.
Apply Lemma \ref{trivial lemma} to the antisymmetric $(\lambda)^{-1} K_p= (aX+bY)^p$, and get $a= -b$, 
hence $R= a(X-Y)$.
Therefore, $K_p= \lambda a^p(X-Y)^p$ and $(E_m)^*+E_m= \mu a^m (X-Y)^m$.
(Actually, $m$ must be even and $\geq 2$, see the following remark \ref{even geq 2}).

In particular, $K_p= \lambda a^p(X-Y)^p$, as claimed.
\end{proof}

\begin{remark}\label{even geq 2}
In the above proof of Theorem \ref{prime degree}, 
it is impossible to have $(E_i)^*+E_i= 0$ for all $i \in \{p-1, p-2, \ldots, 1, 0\}$ 
($E_p$ is antisymmetric, so $(E_p)^*+E_p= 0$). 
Indeed, if $(E_i)^*+E_i= 0$ for all $i \in \{p-1, p-2, \ldots, 1, 0\}$, then
$[f(x), f(y)]= [E_p+E_{p-1}+\ldots+E_1+E_0, (E_p)^*+(E_{p-1})^*+\ldots+(E_1)^*+(E_0)^*]=$
$[E_p+E_{p-1}+\ldots+E_1+E_0, -E_p-E_{p-1}-\ldots-E_1-E_0]=$
$[E_p+E_{p-1}+\ldots+E_1+E_0, -{E_p+E_{p-1}+\ldots+E_1+E_0}]= 0$, a contradiction.

Also, the maximal $j$ ($j \in \{p-1, p-2, \ldots, 1, 0\}$) such that $(E_i)^*+E_i\neq 0$,
 is necessarily $\geq 1$, since otherwise, $(E_j)^*+E_j= 0$ for all $j \in \{p, p-1, p-2, \ldots, 1\}$, 
 hence $[f(x), f(y)]= [E_p+E_{p-1}+\ldots+E_1+E_0, (E_p)^*+(E_{p-1})^*+\ldots+(E_1)^*+(E_0)^*]=$
$[E_p+E_{p-1}+\ldots+E_1+E_0, -E_p-E_{p-1}-\ldots-E_1+E_0]=$
$[E_p+E_{p-1}+\ldots+E_1+E_0, -{E_p+E_{p-1}+\ldots+E_1}+E_0]$.
Let $\tilde{E}= E_p+E_{p-1}+\ldots+E_1$, 
so $[f(x), f(y)]= [\tilde{E}+E_0, -\tilde{E}+E_0]= [\tilde{E}, -\tilde{E}]+ [\tilde{E}, E_0]+ [E_0, -\tilde{E}]+ [E_0, E_0]= 0$, a contradiction.

Moreover, $m$ must be even; otherwise, if we apply \cite[Theorem 1.22(2)]{shape of possible counterexamples} to $-1= [f(x), E^*+E]= [f(x), (E_m)^*+E_m]$, we get (by exactly the same arguments as in the above proof)
$K_p= \lambda a^p(X-Y)^p$ and $(E_m)^*+E_m= \mu a^m (X-Y)^m$.

But $(E_m)^*+E_m= \mu a^m (X-Y)^m$ with $m$ odd is impossible, since $(E_m)^*+E_m$ is symmetric,
 while $\mu a^m (X-Y)^m$ is antisymmetric (it is clear that if $m$ is odd, then $(X-Y)^m$ is antisymmetric).
\end{remark}

\section{An additional result}
The discussion in this section (except for the second subsection: Second idea) relies heavily on results of J.A. Guccione, J.J. Guccione and C. Valqui (brought in \cite{shape of possible counterexamples}) and on a result of Joseph (\cite[Corollary 5.5]{joseph}), as one will clearly see.  

We continue to assume that $\Char(F)= 0$.

Recall the following definition which appears in \cite[Definition 3.1]{shape of possible counterexamples}:

\begin{defin}\label{old definitions} 
Let $f$ be an endomorphism of $A_1$.
$f$ is an \textit{irreducible} endomorphism if the following two conditions are satisfied:\bit

\item $\nu_{1,1}(f(x)) \geq 2$ and $\nu_{1,1}(f(y)) \geq 2$.

\item For every pair of automorphisms $a,b$ of $A_1$: 

$\nu_{1,1}((afb)(x)) + \nu_{1,1}((afb)(y)) \geq \nu_{1,1}(f(x)) + \nu_{1,1}(f(y))$.

\eit 
A pair $(P,Q)$ of elements of $A_1$ is an \textit{irreducible} pair, if there exists an irreducible endomorphism $f$ of $A_1$ such that $P= f(x)$ and $Q= f(y)$.
\end{defin}

We adjust the above definition to our starred setting.

\begin{defin}[An $\alpha$-irreducible $\alpha$-endomorphism]
Let $f$ be an $\alpha$-endomorphism of $A_1$.
$f$ is an \textit{$\alpha$-irreducible} $\alpha$-endomorphism if the following two conditions are satisfied:\bit
\item [(1)] $\nu_{1,1}(f(x)) \geq 2$ and $\nu_{1,1}(f(y)) \geq 2$.
\item [(2)] For every pair of $\alpha$-automorphisms $a,b$ of $A_1$: 
$\nu_{1,1}((afb)(x)) + \nu_{1,1}((afb)(y)) \geq \nu_{1,1}(f(x)) + \nu_{1,1}(f(y))$.
\eit
A pair $(P,Q)$ of elements of $A_1$ is an \textit{$\alpha$-irreducible} pair, if there exists an $\alpha$-irreducible $\alpha$-endomorphism $f$ of $A_1$ such that $P= f(x)$ and $Q= f(y)$.
\end{defin}

Recall the following theorem which is \cite[Theorem 3.3]{shape of possible counterexamples}:
``If there is no irreducible endomorphism, then every endomorphism of $A_1$ is an automorphism".

Similarly, we have:
\begin{thm}\label{starred thm 3.3}
Assume that there are no $\alpha$-irreducible $\alpha$-endomorphisms of $A_1$. Then every $\alpha$-endomorphism of $A_1$ is an ($\alpha$-)automorphism.
\end{thm}

\begin{proof}
Otherwise, there exists an $\alpha$-endomorphism of $A_1$ which is not an automorphism (more accurately, there exists an $\alpha$-endomorphism of $A_1$ which is not onto). 
Let $A$ be the set of all $\alpha$-endomorphisms of $A_1$ which are not onto. 
For each $g \in A$, one can associate the unique natural number $\nu_{1,1}(g(x)) + \nu_{1,1}(g(y))$ 
(since $g$ is an $\alpha$-endomorphism, $\nu_{1,1}(g(x))= \nu_{1,1}(g(y))$).

Of course, the set of those numbers, 
$\{ \nu_{1,1}(g(x)) + \nu_{1,1}(g(y)) | g \in A \}$ has a minimum, denote it by $m$ ($m$ must be even). 
{}From all those $g \in A$ for which $\nu_{1,1}(g(x)) + \nu_{1,1}(g(y))= m$, choose one such, and denote it by $f$. 

By assumption, there are no $\alpha$-irreducible $\alpha$-endomorphisms of $A_1$, hence $f$ (as an $\alpha$-endomorphism) is necessarily $\alpha$-reducible.
This means that the first condition is not satisfied by $f$ or the second condition is not satisfied by $f$ (the conditions in the definition of an $\alpha$-irreducible $\alpha$-endomorphism).

By our special choice of $f$ as an element of $A$ such that for every $h \in A$,
$\nu_{1,1}(f(x)) + \nu_{1,1}(f(y)) \leq \nu_{1,1}(h(x)) + \nu_{1,1}(h(y))$, we get that the second condition must be satisfied by $f$; 
otherwise, there exist $\alpha$-automorphisms $a,b$ of $A_1$ such that 
$\nu_{1,1}((afb)(x)) + \nu_{1,1}((afb)(y)) < \nu_{1,1}(f(x)) + \nu_{1,1}(f(y))= m$.

Claim: $afb \in A$.
Proof of claim: 
\bit 
\item $afb$ is an $\alpha$-endomorphism, since $a$, $f$ and $b$ are $\alpha$-endomorphisms.
\item $afb$ is not onto; otherwise, $t= afb$ is an automorphism, then composing $a^{-1}$ on the left and $b^{-1}$ on the right (remember that $a$ and $b$ are automorphisms), we get $a^{-1}tb^{-1}= f$. But $a^{-1}$, $t$, and $b^{-1}$ are automorphisms, hence $a^{-1}tb^{-1}= f$ is an automorphism, a contradiction, since $f \in A$ ($A$ is the set of all $\alpha$-endomorphisms of $A_1$ which are not onto).
\eit

But $afb \in A$ and $\nu_{1,1}((afb)(x)) + \nu_{1,1}((afb)(y)) < \nu_{1,1}(f(x)) + \nu_{1,1}(f(y))= m$ contradicts the minimality of $m$.

Therefore, necessarily the first condition is not satisfied by $f$, 
so $\nu_{1,1}(f(x)) < 2$ or $\nu_{1,1}(f(y)) < 2$.
(actually, $\nu_{1,1}(f(x))= \nu_{1,1}(f(y))$, 
so $\nu_{1,1}(f(x))= \nu_{1,1}(f(y)) < 2$).

Notice that $\nu_{1,1}(f(x)) \neq 0$, since $\nu_{1,1}(f(x))= 0$ implies that $f(x)= \alpha_{00} \in F$,
which is impossible, because $[f(y), f(x)]= 1$ (and $f(x)= \alpha_{00}$ would imply $[f(y), f(x)]= 0$).

Therefore, $\nu_{1,1}(f(y))= \nu_{1,1}(f(x))= 1$, namely, $f$ is an $\alpha$-endomorphism of degree $1$. But we have mentioned in Remark \ref{degree 1 is auto} that an $\alpha$-endomorphism of degree $1$ is an automorphism, a contradiction to $f \in A$.

Concluding that every $\alpha$-endomorphism of $A_1$ is an automorphism.
\end{proof}

(Observe that in order to prove Theorem \ref{starred thm 3.3} it was necessary to demand in the definition of an $\alpha$-irreducible $\alpha$-endomorphism that $a$ and $b$ are $\alpha$-automorphisms).

It is unknown whether or not an irreducible endomorphism exists. It is also unknown whether or not an $\alpha$-irreducible endomorphism exists.

We move to discuss two ideas concerning the original Dixmier's conjecture: 

\subsection{First idea}
There exists a nice connection between an $\alpha$-endomorphism and a reducible endomorphism. 
\begin{prop}\label{every alpha is reducible}
Every $\alpha$-endomorphism is a reducible endomorphism.
\end{prop}

\begin{proof}
Follows from \cite[Proposition 3.8]{shape of possible counterexamples}.
Shortly, let $f$ be an $\alpha$-endomorphism of $A_1$.
Clearly, $\nu_{1,1}(f(x))= \nu_{1,1}(f(y))$, so the greatest common divisor of $\nu_{1,1}(f(x))$ and $\nu_{1,1}(f(y))$ is $\nu_{1,1}(f(x))= \nu_{1,1}(f(y))$. Hence, \cite[Proposition 3.8]{shape of possible counterexamples} implies that $f$ is reducible.
\end{proof}

In view of Proposition \ref{every alpha is reducible}, it would be nice to find some ``density" theorem concerning endomorphisms and $\alpha$-endomorphisms of $A_1$. If one can somehow show that every endomorphism is a ``limit" of $\alpha$-endomorphisms, and a limit of reducible endomorphisms is also reducible, then from \cite[Theorem 3.3]{shape of possible counterexamples} we get that the original Dixmier's conjecture is true.

\subsection{Second idea}

Assume in this subsection that the starred Dixmier's conjecture is true, namely, every $\alpha$-endomorphism of $A_1(F)$ is an automorphism.
Then we have the following:

\begin{prop}\label{beta gamma}
Let $f$ be an endomorphism of $A_1$. Assume that there exist involutions $\beta$ and $\gamma$, each is conjugate to $\alpha$ by an automorphism, such that $f\beta= \gamma f$ (in other words, $f$ is a $(\beta, \gamma)$-endomorphism). Then $f$ is an automorphism.
\end{prop}

\begin{proof}
$\beta$ is conjugate to $\alpha$ by an automorphism means that there exists $g \in \Aut_F(A_1(F))$
such that $\beta=g^{-1}\alpha g$.
$\gamma$ is conjugate to $\alpha$ by an automorphism means that there exists $h \in \Aut_F(A_1(F))$
such that $\gamma=h^{-1}\alpha h$.
By assumption $f\beta= \gamma f$, hence
$fg^{-1}\alpha g= h^{-1}\alpha hf$.
Then, $h(fg^{-1}\alpha g)g^{-1}= h(h^{-1}\alpha hf)g^{-1}$,
so we have $(hfg^{-1}) \alpha= \alpha (hfg^{-1})$.
This means that $hfg^{-1}$ is an $\alpha$-endomorphism of $A_1$.
Therefore (remember that in this subsection we assume that the starred Dixmier's conjecture is true) $hfg^{-1}$ is an automorphism.
Then clearly $f= 1f1= (h^{-1}h)f(g^{-1}g)= h^{-1}(hfg^{-1})g$ is an automorphism, 
as a product of three automorphisms: $h^{-1}$, $hfg^{-1}$ and $g$.
\end{proof}

\begin{thm}\label{beta gamma dixmier}
If for any endomorphism $f$ of $A_1$, there exist involutions $\beta$ and $\gamma$ (each is conjugate to $\alpha$ by an automorphism) such that $f\beta= \gamma f$, then Dixmier's conjecture is true.
(We continue to assume that the starred Dixmier's conjecture is true).
\end{thm}

\begin{proof}
Let $f$ be an endomorphism of $A_1$.
By Proposition \ref{beta gamma}, $f$ is an automorphism.
\end{proof}

However, we do not know if the condition in Theorem \ref{beta gamma dixmier} is true; 
namely, we do not know if for any endomorphism $f$ of $A_1$, there exist involutions $\beta$ and $\gamma$ 
(each is conjugate to $\alpha$ by an automorphism) such that $f\beta= \gamma f$.

Notice that this condition is equivalent to the following condition:
For any endomorphism $f$ of $A_1$, there exist $g,h \in \Aut_F(A_1(F))$ such that $hfg^{-1}$ is an $\alpha$-endomorphism.
 
\subsection{Third idea}
We suggest to combine Joseph's result \cite[Corollary 5.5]{joseph} and J.A. Guccione, J.J. Guccione and Valqui result \cite[Theorem 5.11]{shape of possible counterexamples}, in order to try to prove that the original Dixmier's conjecture is true.

\cite[Theorem 5.11]{shape of possible counterexamples} says the following: ``Let $(P,Q)$ be an irreducible pair. Then there exist $\mu_P, \mu_Q \in F^*$, 
$a,b,m,n \in \N$ and $g \in \Aut_F(A_1(F))$, such that $m,n > 1$, $\gcd(m,n)= 1$, $1 \leq a < b$ and
$l_{1,1}(g(P))= \mu_P x^{am}y^{bm}$, 
$l_{1,1}(g(Q))= \mu_Q x^{an}y^{bn}$, 
$\nu_{1,1}(g(P))= \nu_{1,1}(P)$,
$\nu_{1,1}(g(Q))= \nu_{1,1}(Q)$.
Furthermore, $g(P)$ and $g(Q)$ are subrectangular and the pair $(g(P),g(Q))$ is irreducible".

\begin{remark}\label{remark 1 irr onto}
In the above irreducible pair $(g(P),g(Q))$,
there exist $\tilde{i},\tilde{j} \in \N$ with $0 \leq \tilde{j} < \tilde{i} \leq am$ such that $x^{\tilde{i}}y^{\tilde{j}}$ appears in $g(P)$,
and there exist $i,j \in \N$ with $0 \leq j < i \leq an$ such that $x^{i}y^{j}$ appears in $g(Q)$.
This follows from \cite[Proposition 3.6]{shape of possible counterexamples}.
\end{remark}
Now, in view of \cite[Theorem 3.3]{shape of possible counterexamples}, one wishes to show that there is no irreducible endomorphism (since then the original Dixmier's conjecture is true).

Actually, only the existence of an irreducible endomorphism which is not onto is problematic, namely: 
\begin{thm}\label{new thm 3.3}
If every endomorphism of $A_1$ is reducible or irreducible which is onto, then Dixmier's conjecture is true.
\end{thm}

\begin{proof}
Otherwise, let $\psi$ be an endomorphism of $A_1$ which is not onto. There are two options:\bit
\item $\psi$ is reducible. Then continue in a similar way as in the proof of \cite[Theorem 3.3]{shape of possible counterexamples} to get a contradiction, namely, to get that $\psi$ is onto (use our assumption that every irreducible endomorphism is onto).
\item $\psi$ is irreducible. By our assumption that each irreducible endomorphism is onto, we get that, in particular, $\psi$ is onto. But we have taken $\psi$ not onto.
\eit
Concluding that every endomorphism of $A_1$ is onto.
\end{proof}

So our aim is to show that every irreducible endomorphism of $A_1$ is onto (Theorem \ref{irr onto}), since then by Theorem \ref{new thm 3.3} the original Dixmier's conjecture is true.

Notations untill the end of this section: $\psi^{(M)}$ is a homomorphism of $A_1$ into $A_1^{(m)}$ with 
$\psi^{(M)}(x)= A= A_l+\epsilon$, $\psi^{(M)}(y)= B= B_l+\delta$, 
where $A_l= l_{1,-1}(A)$ and $B_l= l_{1,-1}(B)$ ($A,B,A_l,B_l,\epsilon,\delta \in A_1^{(M)}$).

$\nu_{1, -1}(A_l)= \tilde{a} \in \Q$, $\nu_{1, -1}(B_l)= \tilde{b} \in \Q$.

\begin{remark}\label{remark 2 irr onto}
Notice that: \bit
\item $\psi^{(M)}$ must be injective, since any homomorphism of $A_1$ into $A_1^{(m)}$ is injective ($A_1$ is simple).
\item $\nu_{1, -1}(\epsilon) < \nu_{1, -1}(A_l)= \tilde{a}$
and $\nu_{1, -1}(\delta) < \nu_{1, -1}(B_l)= \tilde{b}$.
\eit
\end{remark}

\begin{lem}\label{first lemma}
Then there are three options:\bit
\item $\tilde{a}>0$ and $\tilde{b} \leq 0$.
\item $\tilde{a} \leq 0$ and $\tilde{b}>0$.
\item $\tilde{a}>0$ and $\tilde{b}>0$.
\eit
Actually, in any option, we have $\tilde{a}+\tilde{b}\geq 0$.
\end{lem}

\begin{proof}
{}From \cite[Remark 1.14]{shape of possible counterexamples version 3}, $\tilde{a}> 0$ or $\tilde{b}> 0$.
Therefore, the three options are obvious.

$\tilde{a}+\tilde{b}\geq 0$ follows from \cite[Remark 1.13]{shape of possible counterexamples version 3}:
$0= \nu_{1,-1}(1)= \nu_{1,-1}([B,A]) \leq \nu_{1,-1}(B) + \nu_{1,-1}(A)= \tilde{b} + \tilde{a}$.
\end{proof}

When we say that the monomial $x^iy^j$ ($i \in \Q$, $j \in \N$) appears in $D \in A_1^{(M)}$, 
we mean that there exists $c_{ij} \in F^*$ such that $D= \ldots + c_{ij}x^iy^j + \ldots$
($c_{ij}x^iy^j$ does not necessarily belongs to the $(1,-1)$-leading term of $D$).

\begin{lem}\label{second lemma}
If $\tilde{a} > 0$ and $\tilde{b} \leq 0$, then there are two cases:
\bit
\item [(1)] There exists $0 < u \in \Q$ such that
$x^{u+1}y$ appears in $A$ and $x^{-u}$ appears in $B$.
\item [(2)] There exists $0 > v \in \Q$ such that
$x^{v+1}y$ appears in $B$ and $x^{-v}$ appears in $A$.
\eit

If $\tilde{a} \leq 0$ and $\tilde{b} > 0$, then there are two cases:
\bit
\item [(1)] There exists $0 < u \in \Q$ such that
$x^{u+1}y$ appears in $B$ and $x^{-u}$ appears in $A$.
\item [(2)] There exists $0 > v \in \Q$ such that
$x^{v+1}y$ appears in $A$ and $x^{-v}$ appears in $B$.
\eit

\end{lem}

For example, if $\tilde{a}>0$ and $\tilde{b} \leq 0$:
\bit
\item For $u=3$:
$x^{3+1}y=x^4y$ appears in $A$ and $x^{-3}$ appears in $B$.

\item For $v= -3$:
$x^{-3+1}y= x^{-2}y$ appears in $B$ and $x^{-(-3)}= x^3$ appears in $A$.

\item For $v= -1$:
$x^{-1+1}y= y$ appears in $B$ and $x^{-(-1)}= x$ appears in $B$.
\eit

\begin{proof}
Each claim follows from $1= [\psi^{(M)}(y), \psi^{(M)}(x)]$ and \cite{shape of possible counterexamples version 3}[Lemma 1.7].
\end{proof}

If, for example, $\tilde{a} > 0$ and $\tilde{b} \leq 0$, it is not true that there exists $\N \ni l \geq 2$
such that $x^l$ appears in $A$ and $y^l$ appears in $B$; only for $l=1$ it may (or may not) happen that
$x$ appears in $A$ and $y$ appears in $B$.

Indeed, if $\N \ni l \geq 2$ is such that $x^l$ appears in $A$ and $y^l$ appears in $B$,
take $m \geq 2$ be maximal with that property.

Clearly, for any $\N \ni k \geq 1$, we have 
$[y^k,x^k]=c_{k-1}x^{k-1}y^{k-1}+c_{k-2}x^{k-2}y^{k-2}+\ldots+c_2x^2y^2+c_1xy+c_0$,
where $c_i \in F^*$ fol all $0 \leq i \leq k-1$ (the $c_i$'s are easily obtained from the formula in \cite{shape of possible counterexamples version 3}[Lemma 1.7]).

But then $x^{m-1}y^{m-1}$ appears in $[B,A]$, a contradiction to $[B,A]=1$.

Now we bring our attempt to prove that each irreducible endomorphism is onto.

We wish to emphasyze that if our proof of Theorem \ref{irr onto} will be found to be true, then it is obvious that without the results of
Joseph \cite[Corollary 5.5]{joseph}, J.A. Guccione, J.J. Guccione and Valqui \cite[Theorem 5.11]{shape of possible counterexamples}, it seems (at least to me) an impossible mission to prove Dixmier's conjecture.

We guess Dixmier's results \cite{dixmier} and probably other results of additional researchers, should also be considered as necessary ingredients in our proof, since they inspired \cite{joseph} and 
\cite{shape of possible counterexamples}.

So, if our proof is true, then at least six people are responsible for it. If our proof is false, then only one person is to blame.

\begin{thm}\label{irr onto}
Each irreducible endomorphism of $A_1$ is onto.
\end{thm}
 
\begin{proof}
Let $f$ be an irreducible endomorphism of $A_1$. 

Let $P= f(x)$ and $Q= f(y)$. Then, by definition, $(P,Q)$ is an irreducible pair. 
Hence, from \cite[Theorem 5.11]{shape of possible counterexamples} there exist $\mu_P, \mu_Q \in F^*$, 
$a,b,m,n \in \N$ and $g \in \Aut_F(A_1(F))$, such that $m,n > 1$, $\gcd(m,n)= 1$, $1 \leq a < b$ and
$l_{1,1}(g(P))= \mu_P x^{am}y^{bm}$, 
$l_{1,1}(g(Q))= \mu_Q x^{an}y^{bn}$, 
$\nu_{1,1}(g(P))= \nu_{1,1}(P)$,
$\nu_{1,1}(g(Q))= \nu_{1,1}(Q)$,
$g(P)$ and $g(Q)$ are subrectangular and the pair $(g(P),g(Q))$ is irreducible.

So, $\mu_P x^{am}y^{bm}= l_{1,1}(g(P))= l_{1,1}(g(f(x)))= l_{1,1}((gf)(x))$ and 

$\mu_Q x^{an}y^{bn}= l_{1,1}(g(Q))= l_{1,1}(g(f(y)))= l_{1,1}((gf)(y))$.

\textbf{Claim}: $gf$ is an automorphism.

\textbf{Proof of claim}: From Joseph's result \cite[Corollary 5.5]{joseph}, 
it is enough to show that there exist no positive integer $M$ and a map $\psi^{(M)} \in H^{(M)}$
such that for some $\Z \ni l \neq 0, F \ni \beta \neq 0$, 

$l_{1, -1}(\psi^{(M)}((gf)(x))) = -(M/ \beta l) x^{- l/M}$ 

and $l_{1, -1}(\psi^{(M)}((gf)(y))) = \beta x^{1+ l/M} y$. 
Otherwise, let $M$ be a positive integer and $\psi^{(M)} \in H^{(M)}$ a map

such that for some $\Z \ni l \neq 0, F \ni \beta \neq 0$, 

$l_{1, -1}(\psi^{(M)}((gf)(x))) = -(M/ \beta l) x^{- l/M}$ 

and $l_{1, -1}(\psi^{(M)}((gf)(y))) = \beta x^{1+ l/M} y$.

Of course, $\nu_{1, -1}(\psi^{(M)}((gf)(x)))= - \nu_{1, -1}(\psi^{(M)}((gf)(y)))$; namely one $(1,-1)$-degree is positive and the other $(1,-1)$-degree is negative.

We have seen in Lemma \ref{first lemma} that there are three options; in each option we will show that it is impossible to have $\nu_{1, -1}(\psi^{(M)}((gf)(x)))= - \nu_{1, -1}(\psi^{(M)}((gf)(y)))$.

In other words, in each option we will show that it is impossible to have one of 
$\nu_{1, -1}(\psi^{(M)}((gf)(x))), \nu_{1, -1}(\psi^{(M)}((gf)(y)))$ positive and the other negative.

Therefore, Joseph's result \cite[Corollary 5.5]{joseph} would imply that $gf$ is onto.

\textbf{First option} $\tilde{a} > 0$ and $\tilde{b} \leq 0$:
(For example, $A= cx^8+x^2y$ and $B=d+x^{-1}$, where $c,d \in F$. Indeed,
$[B,A]=[d+x^{-1},cx^8+x^2y]=$
$[d,cx^8]+[d,x^2y]+[x^{-1},cx^8]+[x^{-1},x^2y]=$
$0+0+0+[x^{-1},x^2y]=$
$x^{-1}x^2y-x^2yx^{-1}=$
$xy-x^2(x^{-1}y-x^{-2})=$
$xy-xy+1= 1$.
Another example: $A=x$ and $B=x^{-1}+y$).

{}From Lemma \ref{second lemma} we get that there are two cases:
\bit
\item [(1)] There exists $0 < u \in \Q$ such that
$x^{u+1}y$ appears in $A$ and $x^{-u}$ appears in $B$:

Remark \ref{remark 1 irr onto} says that there exist $\tilde{i},\tilde{j} \in \N$ with $0 \leq \tilde{j} < \tilde{i} \leq am$ such that $x^{\tilde{i}}y^{\tilde{j}}$ appears in $(gf)(x)$,
and there exist $i,j \in \N$ with $0 \leq j < i \leq an$ such that $x^{i}y^{j}$ appears in $(gf)(y)$.

Therefore, $A^{\tilde{i}}B^{\tilde{j}}$ appears in $\psi^{(M)}((gf)(x))$
and $A^{i}B^{j}$ appears in $\psi^{(M)}((gf)(y))$.

Hence, $(x^{u+1}y)^{\tilde{i}}(x^{-u})^{\tilde{j}}$ appears in $\psi^{(M)}((gf)(x))$
and $(x^{u+1}y)^{i}(x^{-u})^{j}$ appears in $\psi^{(M)}((gf)(y))$.

Of course,
$\nu_{1,-1}((x^{u+1}y)^{\tilde{i}}(x^{-u})^{\tilde{j}})=$
$u(\tilde{i}-\tilde{j}) > 0$
(so there is a chance that $\nu_{1,-1}(\psi^{(M)}((gf)(x))) > 0$)

and $\nu_{1,-1}((x^{u+1}y)^{i}(x^{-u})^{j})=$
$u(i-j) > 0$

(so there is a chance that $\nu_{1,-1}(\psi^{(M)}((gf)(y))) > 0$).

Now, for $r,s,r',s',R,S,R',S' \in \N$ notice the following two facts; the first is about $r > s$, while the second is about $r\leq s$:

When $r > s$: If $x^ry^s$ appears in $(gf)(x)$,
then $A^{r}B^{s}$ appears in $\psi^{(M)}((gf)(x))$.
$A^{r}B^{s}=$
$(\ldots+x^{u+1}y+\ldots)^{r}(\ldots+x^{-u}+\ldots)^{s}=$
$(\ldots+x^{r(u+1)}y^{r}+\ldots)(\ldots+x^{s(-u)}+\ldots)=$
$\ldots+ x^{r(u+1)}y^{r}x^{s(-u)} +\ldots=$
$\ldots+ x^{r(u+1)-su}y^{r} +\ldots$,
so $x^{r(u+1)-su}y^{r}$ apppears in $A^{r}B^{s}$
($\nu_{1,-1}(x^{r(u+1)-su}y^{r})=r(u+1)-su-r=ru-su=u(r-s) > 0$).

What is important is that for different pairs $(r,s)$ ($r > s$) and $(r',s')$ ($r' > s'$) 

such that $x^ry^s$ appears in $(gf)(x)$, $x^{r'}y^{s'}$ appears in $(gf)(x)$ and $r-s=r'-s'$, we have: 

$x^{r(u+1)-su}y^{r}$ apppears in $A^{r}B^{s}$
(of course, $A^{r}B^{s}$ appears in $\psi^{(M)}((gf)(x))$)

and $x^{r'(u+1)-s'u}y^{r'}$ apppears in $A^{r'}B^{s'}$ 

(of course, $A^{r'}B^{s'}$ appears in $\psi^{(M)}((gf)(x))$),

which are different monomials (although they have the same $(1,-1)$-degree, namely $u(r-s)=u(r'-s')$).

Exactly the same considerations show that if $x^Ry^S$ appears in $(gf)(y)$,
then $A^{R}B^{S}$ appears in $\psi^{(M)}((gf)(y))$.
$A^{R}B^{S}= \dots =$
$\ldots+ x^{R(u+1)-Su}y^{R} +\ldots$,
so $x^{R(u+1)-Su}y^{R}$ apppears in $A^{R}B^{S}$
($\nu_{1,-1}(x^{R(u+1)-Su}y^{R})=R(u+1)-Su-R=Ru-Su=u(R-S) > 0$).

And for different pairs $(R,S)$ ($R > S$) and $(R',S')$ ($R' > S'$) such that $x^Ry^S$ appears in $(gf)(y)$, $x^{R'}y^{S'}$ appears in $(gf)(y)$ and $R-S=R'-S'$,

we have: $x^{R(u+1)-Su}y^{R}$ apppears in $A^{R}B^{S}$ 

(of course, $A^{R}B^{S}$ appears in $\psi^{(M)}((gf)(y))$)

and $x^{R'(u+1)-S'u}y^{R'}$ apppears in $A^{R'}B^{S'}$

(of course, $A^{R'}B^{S'}$ appears in $\psi^{(M)}((gf)(y))$),

which are different monomials (although they have the same $(1,-1)$-degree, namely $u(R-S)=u(R'-S')$).

When $r \leq s$: If $x^ry^s$ appears in $(gf)(x)$,
then $A^{r}B^{s}$ appears in $\psi^{(M)}((gf)(x))$,
and similarly, if $x^Ry^S$ appears in $(gf)(y)$,
then $A^{R}B^{S}$ appears in $\psi^{(M)}((gf)(y))$.

Generally, $A^{r}B^{s}$ and $A^{R}B^{S}$ may or may not contain monomials of positive $(1,-1)$-degree.

Now, in Lemma \ref{first lemma} we have seen that $\tilde{a}+\tilde{b}\geq 0$. 

If $\tilde{a}+\tilde{b}=0$, then it is clear that for any $r,s \in \N$

$\nu_{1,-1}(A^{r}B^{s})= $
$\nu_{1,-1}((A_l)^{r}(B_l)^{s})=$
$r\tilde{a}+s\tilde{b}=$
$r\tilde{a}+s(-\tilde{a})=$
$(r-s)\tilde{a}$.

Therefore, when $r \leq s$, we have 
$\nu_{1,-1}(A^{r}B^{s})= (r-s)\tilde{a} \leq 0$,
and when $r > s$, we have
$\nu_{1,-1}(A^{r}B^{s})= (r-s)\tilde{a} > 0$.

Hence the different (positive $(1,-1)$-degree) monomials obtained above, namely, in the first fact $r > s$, 
will certainly not be cancelled by (non-positive $(1,-1)$-degree) monomials obtained from $r \leq s$, 
so

$\nu_{1,-1}(\psi^{(M)}((gf)(x))) > 0$ and 
$\nu_{1,-1}(\psi^{(M)}((gf)(y))) > 0$.

If $\tilde{a}+\tilde{b} > 0$, the situation is a little more complicated;
however, here also one can obtain 
$\nu_{1,-1}(\psi^{(M)}((gf)(x))) > 0$ and 
$\nu_{1,-1}(\psi^{(M)}((gf)(y))) > 0$. 

Indeed, although $A^{r}B^{s}$ (and $A^{R}B^{S}$) where $r \leq s$ (and $R \leq S$), 
may contain monomials of positive $(1,-1)$-degree,
so apriori such monomials of positive $(1,-1)$-degree may cancel monomials of positive $(1,-1)$-degree that are obtained from $A^{\hat{r}}B^{\hat{s}}$ (and from $A^{\hat{R}}B^{\hat{S}}$), 
where $\hat{r} > \hat{s}$ and $\hat{R} > \hat{S}$.

However, since we consider the case in which $x^{u+1}y$ appears in $A$ and $x^{-u}$ appears in $B$ ($0 < u \in \Q$),

one can see that monomials of positive $(1,-1)$-degree appearing in $A^{r}B^{s}$ (and in $A^{R}B^{S}$) 
where $r \leq s$ (and $R \leq S$), must differ from monomials of positive $(1,-1)$-degree appearing in $A^{\hat{r}}B^{\hat{s}}$ (and in $A^{\hat{R}}B^{\hat{S}}$) where $\hat{r} > \hat{s}$ (and $\hat{R} > \hat{S}$).

\item [(2)] There exists $0 > v \in \Q$ such that
$x^{v+1}y$ appears in $B$ and $x^{-v}$ appears in $A$:

Similarly to the above case (namely, the case in which there exists $0 < u \in \Q$ such that
$x^{u+1}y$ appears in $A$ and $x^{-u}$ appears in $B$),
one can obtain 
$\nu_{1,-1}(\psi^{(M)}((gf)(x))) > 0$ and 
$\nu_{1,-1}(\psi^{(M)}((gf)(y))) > 0$.
\eit

\textbf{Second option} $\tilde{a} \leq 0$ and $\tilde{b} > 0$:
Notice that this option is not symmetric to the first option,
since $am < bm$ and $an < bn$.

It is clear that, if we write $bm= am+t$ and $bn= an+t'$ with $t,t'>0$, then 
$\nu_{1,-1}(A^{am}B^{bm})=$
$(am)\tilde{a}+(bm)\tilde{b}=$
$(am)\tilde{a}+(am+t)\tilde{b}=$
$(am)(\tilde{a}+\tilde{b})+t\tilde{b} > 0$.
And similarly, 
$\nu_{1,-1}(A^{an}B^{bn})=$
$(an)(\tilde{a}+\tilde{b})+t'\tilde{b} > 0$.

As in the first option, from Lemma \ref{second lemma} we get that there are two cases:
\bit
\item [(1)] There exists $0 < u \in \Q$ such that
$x^{u+1}y$ appears in $B$ and $x^{-u}$ appears in $A$.
\item [(2)] There exists $0 > v \in \Q$ such that
$x^{v+1}y$ appears in $A$ and $x^{-v}$ appears in $B$.
\eit

(Observe that now, when $r \leq s$ we get that
$A^{r}B^{s}$ is of non-negative $(1,-1)$-degree,
while in the above First option, when $r \leq s$ we get that
$A^{r}B^{s}$ is of non-positive $(1,-1)$-degree).

One can see that, in each of those two cases, monomials of positive $(1,-1)$-degree appearing in $A^{am}B^{bm}$ 
(and in $A^{an}B^{bn}$) must differ from monomials of positive $(1,-1)$-degree appearing in 
$A^{r}B^{s}$ (and in $A^{R}B^{S}$) where: 

$r \leq s$ (and $R \leq S$) and $A^{r}B^{s}$ appears in $\psi^{(M)}((gf)(x))$
(and $A^{R}B^{S}$ appears in $\psi^{(M)}((gf)(y))$).

So,
$\nu_{1,-1}(\psi^{(M)}((gf)(x))) > 0$ and 
$\nu_{1,-1}(\psi^{(M)}((gf)(y))) > 0$.

\textbf{Third option} $\tilde{a} > 0$ and $\tilde{b} > 0$:
It is clear that
$l_{1, -1}(\psi^{(M)}((gf)(x))) = (A_l)^{am}(B_l)^{bm}$

and $l_{1, -1}(\psi^{(M)}((gf)(y))) = (A_l)^{an}(B_l)^{bn}$.

So, $\nu_{1, -1}(\psi^{(M)}((gf)(x)))= (am)\tilde{a} + (bm)\tilde{b} > 0$

and $\nu_{1, -1}(\psi^{(M)}((gf)(y)))= (an)\tilde{a} + (bn)\tilde{b} > 0$.
Finally, if our claim that $gf$ is an automorphism is indeed true, then since $g$ is an automorphism, we get that $f$ is an automorphism, because $f= 1f= (g^{-1}g)f= g^{-1}(gf)$ is a product of two automorphisms: $g^{-1}$ and $gf$.
\end{proof}

\section{Related topics}
We suggest to consider the following topics. In those topics we (usually) took the exchange involution $\alpha$, although one may take other involutions as well.

\subsection{Prime characteristic case}

When $F$ is of prime characteristic, $A_1(F)$ is not simple. Bavula asked the following question: Is every algebra endomorphism of the first Weyl algebra $A_1(F)$, $\Char(F)= p > 0$, a monomorphism? Makar-Limanov \cite{makar BC} gave a positive answer to that question. However, according to \cite{normal subgroup}, the following endomorphism $f$ 
($f$ is necessarily a monomorphism) is not onto, since $x$ is not in the image of $f$:

$f: A_1 \longrightarrow A_1$ such that $f(x)= x+x^p$ and $f(y)= y$. 

Since this $f$ is not an $\alpha$-endomorphism, one may wish to ask the following question: 

\begin{question} 
Is every $\alpha$-endomorphism of $A_1(F)$, $\Char(F)= p > 0$, onto? 
\end{question}

More generally, Makar-Limanov gave a negative answer to the following question:
Is every algebra endomorphism of the $n$'th Weyl algebra $A_n(F)$, $n\geq 2$, $\Char(F)= p > 0$, a monomorphism?
Namely, there exists an algebra endomorphism of the $n$'th Weyl algebra $A_n(F)$, $n \geq 2$, $\Char(F)= p > 0$, which is not a monomorphism.

One may wish do define an involution $\alpha$ on $A_n(F)$ ($n\geq 2$, $\Char(F)= p > 0$), an $\alpha$-endomorphism of $A_n(F)$ and generalizations of $\nu_{1,1}$ and of $\nu_{1,-1}$.
Then one may check carefully Makar-Limanov's results and see if his results are applicable to $\alpha$-endomorphisms of $A_n(F)$ or not.
If not, then one may wish to ask the following question: 
\begin{question}
Is every $\alpha$-endomorphism of $A_n(F)$, $n\geq 2$, $\Char(F)= p > 0$, a monomorphism?
\end{question}

\subsection{Weyl algebras}
Let $\Char(F)= 0$. In the second section we have tried to answer the ``starred Dixmier's question", namely: Is every $\alpha$-endomorphism of $A_1(F)$, $\Char(F)= 0$, an automorphism?

More generally, one may wish to define an involution $\alpha$ on $A_n(F)$ ($\geq 2$), an $\alpha$-endomorphism of $A_n(F)$ and degrees.

Then, of course, one may ask the following question:

\begin{question}
Is every $\alpha$-endomorphism of $A_n(F)$, $n \geq 2$, $\Char(F)= 0$, an automorphism?
\end{question}

Call ``the $n$'th starred Dixmier's conjecture" (or the ``$n$'th $\alpha$-Dixmier's conjecture") the conjecture that every $\alpha$-endomorphism of $A_n(F)$ ($\Char(F)= 0$ and $n \geq 1$) is an automorphism. For $n=1$ this is just our starred Dixmier's conjecture (or our $\alpha$-Dixmier's conjecture). Denote it by $\alpha-D_n$.

\subsection{Connection to the Jacobian conjecture}
Let $\Char(F)=0$.
There is an interesting connection between Dixmier's problem $1$ and the Jacobian conjecture.
For a detailed background on the Jacobian conjecture, see, for example, \cite{bass connell wright} or \cite{jacobian van den essen}.

Let $\gamma \in \End_F(F[x_1,\ldots,x_n])$. If $\gamma(x_1)=g_1,\ldots,\gamma(x_n)=g_n$, then we write $\gamma=(g_1,\ldots,g_n)$. 
By definition, the Jacobian matrix of $\gamma$ is $J(\gamma)=(\partial g_i / \partial x_j)$, 
$1 \leq i,j \leq n$.

The Jacobian conjecture-$n$, denoted by $JC_n$ says the following: 

If $\gamma \in \End_F(F[x_1,\ldots,x_n])$ 
with $\det(J(\gamma)) \in F^*$ ($F^*= F-{0})$, 

then $\gamma \in \Aut_F(F[x_1,\ldots,x_n])$ (namely, $\gamma$ is invertible).

The inverse implication is well known, namely: Given $\gamma \in \Aut_F(F[x_1,\ldots,x_n])$, 
then $\det(J(\gamma)) \in F^*$ (see, for example, \cite[page 355]{cohn} or \cite[Appendix 6A]{rowen}).

Notice that when $n=1$: If $\End_F(F[x_1]) \ni \gamma= g_1$ ($g_1 \in F[X_1]$) with $\det(J(\gamma)) \in F^*$, 
then $g_1= aX_1+b$ ($F \in a \neq 0$, $F \ni b$), which is obviously invertible - 
its inverse is $\delta= (1/a)X_1-(1/a)b$.

The connection between Dixmier's problem $1$ and the Jacobian conjecture is as follows:
\begin{itemize}
\item [(1)] The $n$'th Dixmier's conjecture, $D_n$ 

(=Every endomorphism of $A_n(F)$, $\Char(F)= 0$, is an automorphism), 

implies $JC_n$, the Jacobian conjecture-$n$, see \cite[Proposition 3.28]{jacobian van den essen}.

\item [(2)] $JC_{2n}$ $\Rightarrow$ $D_n$: 

Interestingly, the Jacobian conjecture-$2n$ implies the $n$'th Dixmier's conjecture. 
This was proved independently by Tsuchimoto \cite{tsuchimoto} and by Belov and Kontsevich \cite{belov}. 
A shorter proof can be found in \cite{bavula jacobian}.
\end{itemize}

One may ask what happens in the starred case, namely: 
\begin{question}
What should be the $\alpha$-Jacobian conjecture-$n$, $\alpha-JC_n$, such that the following connections hold (maybe there exists no such $\alpha-JC_n$?): \bit

\item [(1)] The $n$'th $\alpha$-Dixmier's conjecture, $\alpha-D_n$ 

(= Every $\alpha$-endomorphism of $A_n(F)$, $\Char(F)= 0$, is an automorphism), 

implies $\alpha-JC_n$.

\item [(2)] $\alpha-JC_{2n}$ $\Rightarrow$ $\alpha-D_n$.
\eit
\end{question}

One may also try to find a connection between the $\alpha$-Dixmier's conjecture and the Poisson conjecture. For details on the Poisson conjecture see, for example, \cite{adja}.

\subsection{The group of automorphisms}
It would be interesting to describe the group of $\alpha$-automorphisms of $A_1(F)$ 

(of course, it is a subgroup of the group of automorphisms of $A_1(F)$) for both $\Char(F)= 0$ and $\Char(F)= p > 0$.

Recall the following results concerning the group of automorphisms of $A_1$; each may be studied in the context of $\alpha$-automorphisms: \bit

\item Dixmier's result (see \cite[Theorem 8.10]{dixmier}) that the group of automorphisms of $A_1(F)$ ($\Char(F)= 0$) is generated by 
$f_{n,\lambda}$ and $f'_{n,\lambda}$,
where $f_{n,\lambda}(x)= x+ \lambda y^n$, $f_{n,\lambda}(y)= y$,
$f'_{n,\lambda}(x)= x$, $f'_{n,\lambda}(y)= y+ \lambda x^n$ ($n \in \N$, $\lambda \in F$). 

It would be nice if the group of $\alpha$-automorphisms of $A_1(F)$ ($\Char(F)= 0$) is generated, as a group with composition as the binary operation, by the family of $\alpha$-automorphisms of Proposition \ref{family of alpha auto}. We do not know yet if this is indeed true (it seems too restrictive to have that family as generators). 

\item Alev's result that the group of automorphisms of $A_1$ is amalgamated, see \cite{alev}.
The group of automorphisms of the polynomial ring $F[x,y]$ ($F$ is a field) is an amalgamated group, see \cite[Theorem 3]{warren}.

\item Makar-Limanov's new proof \cite{makar} of a theorem already brought by Dixmier in \cite{dixmier}, which shows that the group of automorphisms of $A_1$ is isomorphic to a particular subgroup of the group of automorphisms of $F[X,Y]$, the commutative polynomial algebra in two variables.
\eit

\subsection{Same questions for other algebras}
One may wish to ask similar questions for other algebras, see, for example, \cite{bavula analogue}.
In algebras where an involution can be defined, one may wish to see if the presence of an involution may be of any help in solving such questions.

\bibliographystyle{plain}

\end{document}